\documentclass[a4paper,12pt]{article}
\usepackage{amsmath,amsthm,amssymb,amscd}
\usepackage{graphicx}
\usepackage{graphics}
\usepackage{verbatim}
\usepackage{enumerate}
\usepackage{mathrsfs}
\usepackage{color}
\usepackage[top=1in, bottom=0.8in, left=0.8in, right=0.9in]{geometry}
\usepackage[section]{placeins}

\newcommand\ch{\mathrm{ch}}
\makeatletter

\newcommand{\Rmnum}[1]{\expandafter\@slowromancap\romannumeral #1@}

\usepackage[colorinlistoftodos,prependcaption,textsize=scriptsize, textwidth=20mm,]{todonotes}

\usepackage{color}
\def \red {\textcolor{red} }

\newtheorem{theorem}{Theorem}[section]

\newtheorem{lemma}[theorem]{Lemma}
\newtheorem{definition}[theorem]{Definition}
\newtheorem{corollary}[theorem]{Corollary}
\newtheorem{claim}[theorem]{Claim}

\newtheorem{remark}[theorem]{Remark}
\makeatother

\newcommand{\intt}{{\rm int}}
\newcommand{\ext}{{\rm ext}}

\begin{document}
	
	\title{Planar graphs having no cycle of length  $4$, $6$ or $8$ are DP-3-colorable}

	\vspace{3cm}
	\author{Ligang Jin\footnotemark[1]~,
 Yingli Kang\footnotemark[2]~,  Xuding Zhu\footnotemark[1]}

\footnotetext[1]{School of Mathematical Sciences,
	Zhejiang Normal University, Yingbin Road 688,
	321004 Jinhua,
	China; 
	ligang.jin@zjnu.cn (Ligang Jin), xdzhu@zjnu.edu.cn (Xuding Zhu)}

	\footnotetext[2]{Department of Mathematics, Jinhua University of Vocational Technology, Western Haitang Road 888, 321017 Jinhua, China; ylk8mandy@126.com (Yingli Kang)}

	\date{}
	
	\maketitle

\begin{abstract}
The concept of DP-coloring of graphs was introduced by Dvo\v{r}\'{a}k and Postle, and was used to prove that planar graphs without cycles of length from $4$ to $8$ are $3$-choosable. In the same paper, they proposed a more natural and stronger claim that such graphs are DP-$3$-colorable. This paper confirms that claim by proving a stronger result that planar graphs having no cycle of length $4$, $6$ or $8$ are DP-3-colorable. 
\end{abstract}

\textbf{Keywords}: DP-coloring; Erd\H{o}s problem; $S_3$-signed graphs; reducible configurations; discharging

\section{Introduction}

 Steinberg \cite{Steinberg1993211} conjectured in 1976 that   every planar graph without cycles of length 4 or 5 is 3-colorable. This conjecture received considerable attention and was finally disproved in 2017 by Cohen-Addad,   Hebdige,   Kr\'{a}l,  Li, and  Salgado \cite{CA-disprove}.

Motivated by Steinberg's conjecture, Erd\H{o}s  asked  whether  there exists a constant  $k$ such that every planar graph without cycles of length from 4 to $k$ is 3-colorable? If so, what is the smallest constant $k$?
It was proved by
Abbott and Zhou \cite{AbbottZhou1991203}  that such an integer exists and $k\leq 11$. 
The upper bound on the smallest constant $k$ was   improved in a sequence of papers, and the current known best upper bound is  
$k\leq 7$, obtained by   Borodin, Glebov,  Raspaud, and   Salavatipour \cite{4567}. It remains open whether $k=6$ or $k=7$.

 Erd\H{o}s' problem has a natural list coloring version, that has also been studied extensively in the literature.   Abbott and Zhou \cite{AbbottZhou1991203}  actually showed that planar graphs without cycles of length $4$ to $11$ are $3$-choosable, and the question is to find the smallest integer $k'$ such that planar graphs without cycles of length  from $4$ to $k'$ are 3-choosable. It was proved by Voigt \cite{Voigt2007} that $k' \ge 6$ and the following result of 
 Dvo\v{r}\'{a}k  and Postle  gives the current 
  known best upper bound for $k'$. 
  
  \begin{theorem}[\cite{Dvorak-Postle-2018}]
      \label{thm-dp}
      Every planar graph without cycles of length from 4 to 8 is 3-choosable. 
  \end{theorem}

Theorem \ref{thm-dp} answers a  question   posed by Borodin  \cite{Borodin2013} in 1996. 
Nevertheless, it turns out that the more important impact of the work of Dvo\v{r}\'{a}k  and Postle    is   the concept of DP-coloring introduced in \cite{Dvorak-Postle-2018}, which has attracted considerable attention and  has motivated a lot of research. 

  \begin{definition}
	\label{def-cover}
	A {\em  cover} of a  graph $G $ is a pair $(L,M)$, where $L = \{L(v)\colon v \in V(G)\}$ is a family of  disjoint sets, and $M=\{M_{e}\colon e \in E(G)\}$, where for each edge $e=uv$, $M_e$ is a matching between $L(u)  $ and $L(v)$.   For  $f \in \mathbb{N}^G$, we say $(L,M)$ is an {\em $f$-cover} of $G$ if $|L(v)|\ge f(v)$ for each vertex $v \in V(G)$.
\end{definition}

\begin{definition}
	\label{def-coloring}
	Given a cover $(L,M)$ of a graph $G$, an {\em $(L,M)$-coloring} of $G$ is a mapping $\phi\colon V(G) \to \bigcup_{v \in V(G)}L(v)$ such that for each vertex $v \in V(G)$, $\phi(v) \in L(v)$, and for each edge $e=uv \in E(G)$, $\phi(u)\phi(v) \notin E(M_e)$. We say $G$ is {\em $(L, M)$-colorable} if it has an $(L,M)$-coloring. 
   We say $G$ is {\em DP-$f$-colorable}, where  $f \in  \mathbb{N}^G$,   if for every   $f$-cover $(L,M)$, $G$ has an $(L,M)$-coloring. The {\em  DP-chromatic number} of $G$ is defined as 
 $$\chi_{DP}(G)=\min\{k\colon G \text{ is DP-$k$-colorable} \}.$$
\end{definition}

Given an $f$-list assignment $L'$ of a graph $G$, let $(L,M)$ be the   $f$-cover of $G$, where $L(v)=\{(i,v)\colon i \in L'(v)\}$ and  $M=\{M_{uv}\colon uv \in E(G)\}$, where for each edge $uv$ of $G$,  $$M_{uv}=\{\{(i,u),(i,v)\}\colon i \in L'(u) \cap L'(v)\}.$$
It is obvious that $G$ is $L'$-colorable if and only if $G$ is $(L,M)$-colorable. Therefore  if $G$ is DP-$f$-colorable, then it is $f$-choosable, and hence $ch(G) \le \chi_{DP}(G)$.

The advantage of transforming a list coloring problem to a DP-coloring problem is that the information of the lists are encoded in the matchings $M_e$ for edges $e \in E(G)$. There are tools one can use in the study of DP-coloring of graphs that are not applicable in the setting of list coloring. For example, one can identify non-adjacent vertices in the study of DP-coloring, that is not applicable in the study of list coloring. It is by using such tools that  Dvo\v{r}\'{a}k  and Postle were able to prove that planar graphs without cycles of length from $4$ to $8$ are 3-choosable.

As DP-3-colorable graphs are 3-choosable, and DP-coloring technique is used to prove that planar graphs without cycles of length from 4 to 8 are 3-choosable, it seems  more natural to prove that these planar graphs are DP-3-colorable. However, in the proof in \cite{Dvorak-Postle-2018}, the family of matchings $M_e$ are restricted to be  consistent on
closed walks of length 3. This is enough to conclude that the graphs in concern are 3-choosable, but not enough to conclude that  they are   DP-3-colorable. 
It was proved in  \cite{Liu-loeb-Yin-Yu}  that every planar graph without cycles of length from $\{4,5,6,9\}$ is DP-3-colorable. The problem whether all planar graphs without cycles of length from 4 to 8 are DP-3-colorable is proposed in \cite{Dvorak-Postle-2018} and remains open.

In this paper, we solve this problem and prove the following result.

 \begin{theorem} \label{thm468-DP}
 	Every planar graph having no cycle of length 4, 6 or 8 is DP-3-colorable.
 \end{theorem}

\section{$S_3$-signed graphs and configurations}

 We denote by $\mathcal{G}$ the family of connected plane graphs having no cycle of length $4$, $6$ or $8$. We shall prove the following result that implies Theorem \ref{thm468-DP}.

\begin{theorem}   \label{thm_main_extension}
	Assume $G \in \mathcal{G}$  with infinite face $f_0$, and $(L,M)$ is a 3-cover of $G$.  If the boundary of $f_0$ has length at most 12, then every $(L,M)$-3-coloring of $G[V(f_0)]$  extends to an $(L,M)$-3-coloring of $G$.
\end{theorem}

We may assume that $L(v)=\{1,2,3\} \times \{v\}$ for each vertex $v$, and assume that for each edge $uv$ of $G$, $M_{uv}$ is a perfect matching between $L(u)$ and $L(v)$.
For each edge $uv$, the matching $M_{uv}$ can be represented by a permutation $\sigma_{(u,v)}$ of $\{1,2,3\}$, defined as $\sigma_{(u,v)}(i) = j$ if $(i,u)(j,v) \in M_{uv}$. Instead of a family $M=\{M_{uv}\colon uv \in E(G)\}$ of matchings, we have a family $  \sigma=\{\sigma_{(u,v)}\colon uv \in E(G)\}$ of permutations of $\{1,2,3\}$, satisfying  $\sigma_{(v,u)}= \sigma_{(u,v)}^{-1}$. An $(L,M)$-coloring of $G$ is equivalent to a mapping $\phi\colon V(G) \to \{1,2,3\}$ such that for each edge $uv$, $\phi(v) \ne \sigma_{(u,v)}(\phi(u))$. We call the pair $(G, \sigma)$ an {\em $S_3$-signed graph} (where $\sigma_{(u,v)}$ is viewed as a sign of the arc $(u,v)$), and call the mapping $\phi$ a proper {\em 3-coloring} of $(G,\sigma)$.

The following theorem is an equivalent formulation of Theorem \ref{thm_main_extension}.

\begin{theorem}   \label{thm_main_extension-signed}
    Let $(G,\sigma)$ be an $S_3$-signed graph with $G \in \mathcal{G}$.  If the boundary of the infinite face $f_0$ has length at most 12, then every 3-coloring of $(G[V(f_0)],\sigma)$  extends to a 3-coloring of $(G,\sigma)$.
\end{theorem}

The proof of Theorem \ref{thm_main_extension-signed} is by induction.  Assume $(G,\sigma)$ is a counterexample, with $|V(G)|$ minimum. 
We first prove that a family of configurations are reducible, i.e., none of them can  be contained in $G$. In the next section,   by using discharging method, we derive a contradiction.

In the remainder of the paper, we consider proper 3-colorings of $S_3$-signed graphs $(G, \sigma)$ for $G \in \mathcal{G}$ (which is equivalent to DP-3-colorings of graphs $G \in \mathcal{G}$). For convenience, we may denote an  $S_3$-signed graphs $(G, \sigma)$ by $G$. The signs $\sigma(e)$ of edges in $G$ are specified   when needed.

By a {\em switching} at a vertex $v$, we mean choose a permutation $\tau$ of $\{1,2,3\}$, and for each edge $uv$ incident to $v$, 
replace   $\sigma_{(u,v)}$ with $\tau \circ \sigma_{(u,v)}$ 
(and hence replace  $\sigma_{(v,u)}$ with $\sigma_{(v,u)}\circ \tau^{-1}$).
Switching at a vertex $v$ just changes the names of colors for $v$, and does not change the colorability of $G$.

An edge $uv$ is called {\em straight} if $\sigma_{(u,v)} =id$, where $id$ is the identity permutation. 
\begin{remark}
\label{remark1}
    For any set $E'$ of edges that induces an acyclic subgraph of $G$, by applying some switchings, if needed, we may assume that all edges in $E'$ are straight. 
\end{remark}

A   cycle $C$ is called {\em positive} if there exist a sequence of switchings that make all edges in $C$ straight.  Otherwise,  $C$ is {\em negative}.

A vertex on the boundary of $f_0$ is an {\it external} vertex, and other vertices are {\it internal}. 
Denote by $|P|$ the length of a path $P$ (which is the number of edges of $P$), $|C|$ the length of a cycle $C$, and $d(f)$ the size of a face $f$.  
A \textit{$k$-vertex} (resp., $k^+$-vertex and $k^-$-vertex) is a vertex $v$ with $d(v)=k$ (resp., $d(v)\geq k$ and $d(v)\leq k$). 
The notations of $k$-path, $k$-cycle, $k$-face etc. are defined similarly.  
A $k$-cycle with vertices $v_1,\ldots,v_k$ in  cyclic order is denoted by $[v_1\ldots v_k]$.
Let $d_1,d_2,d_3$ be three integers with $3\leq d_1 \leq d_2 \leq d_3$. A {\it $(d_1,d_2,d_3)$-face} is a 3-face $[v_1v_2v_3]$ such that $v_i$ is an internal $d_i$-vertex for  $i\in\{1,2,3\}$.

\begin{itemize}
\item A {\it bad vertex} is a vertex incident with a positive $(3,3,3)$-face. 
\item  A  {\it $3_{\Delta}$-vertex}  is an internal 3-vertex incident with a 3-face.
\item A {\it $4_{\bowtie}$-vertex}   is an internal 4-vertex incident with two non-adjacent 3-faces.
\item A {\em $C$-vertex} is a vertex which is neither a 2-vertex nor a $3_{\Delta}$-vertex nor a $4_{\bowtie}$-vertex.
\end{itemize}

Assume $u$ is a $3_{\Delta}$-vertex on a 3-face $f$. The neighbor of $u$ not on $f$ is called the {\it outer neighbor} of $u$ (also of $f$). We say $u$ is 
\begin{enumerate}
    \item a {\it $3_{\Delta^+}$-vertex} if $f$ is positive and contains at least two $3_{\Delta}$-vertices;
    \item a {\it $3_{\Delta^-}$-vertex} if $f$ is negative and contains at least two $3_{\Delta}$-vertices;
    \item a {\it $3_{\Delta^{\circ}}$-vertex} if $u$ is the only $3_{\Delta}$-vertex on $f$.
\end{enumerate}  
 A {\it $3_{\Delta^{\star}}$-vertex} is a $3_{\Delta^{\circ}}$-vertex whose outer neighbor is not a bad vertex.

Two vertices $u$ and $v$ are {\it $\bowtie$-connected} if there exists a $u$-$v$-path whose interior vertices are all $4_{\bowtie}$-vertices.

Let $f_1,\dots,f_k$ be internal 3-faces. The union $S=\bigcup_{i=1}^kf_i$  is called a {\it snowflake} if  the following  hold:
\begin{enumerate}[(1)]
	\setlength{\itemsep}{0pt}
	\item For each $4_{\bowtie}$-vertex $w$ of $S$, both  3-faces containing $w$ belong to $S$;
	\item Any two nonadjacent vertices of $S$ are $\bowtie$-connected.
\end{enumerate}

Note that two snowflakes of $G$  may share vertices but they are edge-disjoint.
For each snowflake $S$, denote by $3_{\Delta^+}(S)$ the set of $3_{\Delta^+}$-vertices of $S$ and similarly, we define $3_{\Delta^-}(S)$, $3_{\Delta^{\circ}}(S)$, $3_{\Delta^{\star}}(S)$, $3_{\Delta}(S)$, $4_{\bowtie}(S)$, and $C(S)$. 
 Let $T(S)$ be the set of 3-faces of $S$, and for $u\in C(S)$, let $t(S,u)$ denote the number of 3-faces of $S$ containing $u$. Let
 \begin{itemize}
     \item 
 $C_1(S)=\{v \in C: v \text{ is an external $3$-vertex or an external $4$-vertex incident with two} \\ \text{non-adjacent 3-faces} \}.$
 \item  $C_2(S)=C(S)\setminus C_1(S)$. 
 \item  
For $i\in \{1,2\}$, $t_i(S)=\sum_{u\in C_i(S)} t(S,u)$.
\end{itemize}

\begin{definition}
    \label{def-conf}
A {\em configuration} is a 4-tuple
$\mathcal{H}=(H, \tau, \theta, Z)$ such that  $(H, \tau)$ is an $S_3$-signed plane graph,  $\theta$ is a mapping $ V(H) \to \mathbb{N} \cup \{\star\}$, and $Z$ is a subset of $V(H)$. We say an $S_3$-signed plane graph $(G, \sigma)$ contains $(H, \tau, \theta, Z)$ as a configuration if $(H, \tau)$ is   an induced  $S_3$-signed subgraph (with the same plane embedding), vertices in $Z$ are internal vertices, and for $v\in V(H)$, $  
d_G(v) =   \theta(v)$   if $\theta(v)$ is an integer.
If $\theta(v) = \star$, then there is no restriction on $d_G(v)$.

   A configuration $(H, \tau, \theta,Z)$ is called {\em reducible} if  any minimal counterexample $(G, \sigma)$ does not contain configuration $(H, \tau, \theta, Z)$. 
\end{definition}

Note that vertices of $H$ not in $Z$ can be either internal or external vertices.

If $(G, \sigma)$ contains a configuration
$(H, \tau, \theta, Z)$, we say $(G, \sigma)$ is the {\em host signed graph} of $(H, \tau, \theta, Z)$.
A vertex $v$ in a configuration $(H, \tau, \theta, Z)$ is called a $k$-vertex if $\theta(v)=k$, i.e., $v$ has degree $k$ in the host graph. 

We shall often represent a configuration by a figure, and the value $\theta(v)$ and the set $Z$ is indicated by the ``shape'' of the vertex $v$: a solid triangle, a solid square, and a solid circle stands for an internal 3-vertex, an internal 4-vertex, and an arbitrary vertex, respectively. There will be no other type of vertices. The signs $\tau$ on a set of acyclic edges  is irrelevant, as by switching we may assume all the edges are straight. The sign of a cycle (usually a triangle)  is labelled by P, if all edges can be made  straight by a switching, or N otherwise. An unlabelled triangle means that it can be either positive or negative. So the signature $\tau$ is omitted and only some triangles are labelled by P or N. 
The embedding of a configuration is also important (it matters if a path is on the boundary of a face or not), which will be indicated in the figure.

If the mappings $\tau$, $\theta$  and $Z$ are clear from the context, we simply call $H$ a configuration.

\begin{definition}
For $k \ge 1$, 
\begin{enumerate}[(1)]
    \item $I_k$ is the configuration consisting of 
$k$ negative triangles $T_i=[u_{i-1}u_iw_i]$ ($i=1,2,\ldots, k$), where $w_1,w_2, \ldots, w_k, u_k$ are internal 3-vertices  and $u_1,u_2,\ldots, u_{k-1}$ are internal 4-vertices. 
The vertex $u_0$ is an arbitrary vertex, and is called the {\em port} of $I_k$. 
\item $J_k$ is the configuration consisting of $2k$ triangles 
$T_i=[u_{i-1}u_iw_i]$ and $T'_i=[x_iy_iz_i]$ ($i=1,2,\ldots, k$),  where each $T'_i$ is positive, $w_i$ is adjacent to $z_i$, vertices $w_i,x_i,y_i,z_i$ are internal 3-vertices, and $u_1,u_2, \ldots, u_{k-1}$ are internal 4-vertices. The vertices $u_0$ and $u_k$ are arbitrary vertices, and are called the two {\em ports} of $J_k$. 
\end{enumerate}    
\end{definition}

\begin{figure}[h]
	\centering
	\includegraphics[width=15cm]{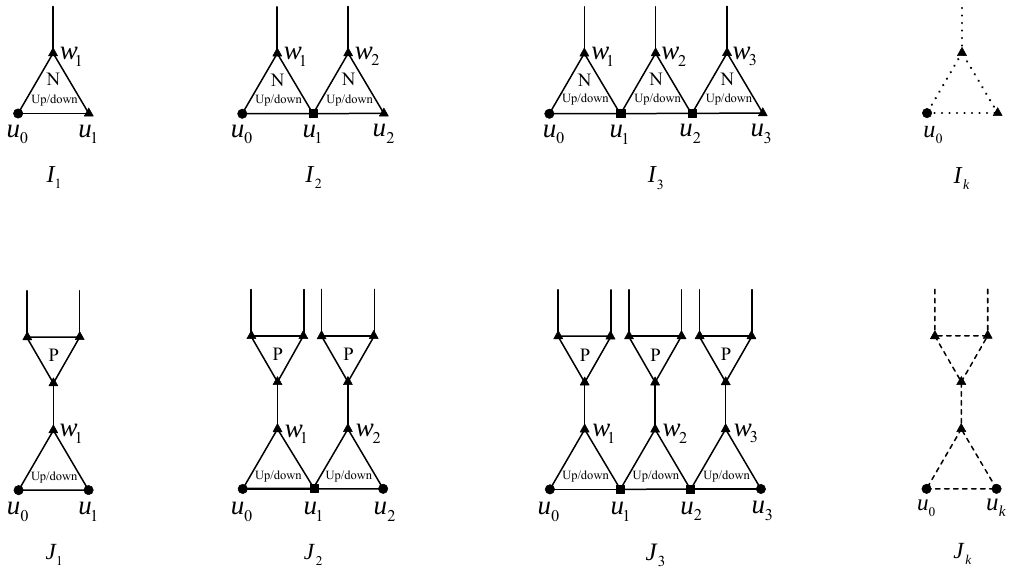}
	\caption{Configurations $I_k$ and $J_k$.  ``Up/down'' indicates that in the plane, the triangle may locate in either side of the path $u_0u_1\ldots u_k$. }\label{fig-IJ-segment}
\end{figure}

 Figure \ref{fig-IJ-segment} are Configurations $I_1,I_2,I_3,J_1,J_2,J_3$ and drawings of general $I_k,J_k$.
 
Assume $(H, \tau, \theta, Z)$ is  a configuration and $v \in V(H)$ with $\theta(v) = d_H(v)+1$ (i.e., $v$ has one neighbor in $V(G)\setminus V(H)$).  Let $H'$ be obtained from the disjoint union of $H$ and $I_k$ by identifying $v$ with the port of $I_k$, and let $\theta'(v) = \theta(v)+1 = d_{H'}(v)$. For other vertices $u$ of $H'$ and edges $e$, $\tau'(e)$ and $\theta'(u)$ and status of $u$ are inherited from $H$ or $I_k$. Then the configuration $(H',\tau',\theta', Z')$ is called the {\em  $I_k$-extension of $H$ at $v$}, and is denoted as $H_{v,I_k}$.

Assume $(H, \tau, \theta, Z)$ is a configuration and $v \in V(H)$ with $\theta(v) = d_H(v)=4$ and $v$ is incident with two non-adjacent triangles, say $T_1$ and $T_2$.  Let $H'$ be obtained from the disjoint union of $H$ and $J_k$ by splitting $v$ into two vertices $v_1$ and $v_2$ so that each $v_i$ is incident with $T_i$ and identifying each $v_i$ with a port of $J_k$, and let $\theta'(v_1) = \theta'(v_2) =4$. For other vertices $u$ of $H'$ and edges $e$, $\tau'(e)$ and $\theta'(u)$ and status of $u$   are inherited from $H$ or $J_k$. Then the configuration $(H',\tau',\theta', Z')$ is called the {\em  $J_k$-extension of $H$ at $v$}, and is denoted as $H_{v,J_k}$.

\begin{definition}
    \label{def-9face}
    Assume  $f=[v_1v_2\ldots v_9]$ is a 9-face of $G$, and $x$ is a vertex of $f$. We say $f$ is a {\em nice $9$-face} and $x$ is a {\em nice vertex of $f$} if  one of the following holds:
\begin{enumerate}[(1)]
	\setlength{\itemsep}{0pt}
	\item $v_2,v_3,v_4,v_5$ are $4_{\bowtie}$-vertices, $v_1$ and $v_6$ are $3_{\Delta}$-vertices, and either $x=v_8$ is a $5^+$-vertex or   $x\in\{v_7,v_9\}$ is a $4^+$-vertex or an external 3-vertex;
	\item $v_4$ is a $4_{\bowtie}$-vertex, $v_3$ and $v_5$ are $3_{\Delta}$-vertices, $v_1,v_2,v_6,v_7$ are bad vertices, and   $x\in\{v_8,v_9\}$ is a $4^+$-vertex or an external 3-vertex;
	\item $v_4$ and $v_5$ are $4_{\bowtie}$-vertices, $v_3$ and $v_6$ are $3_{\Delta}$-vertices, $v_1,v_2,v_7,v_8$ are bad vertices, and $x=v_9$ is a $4^+$-vertex or an external 3-vertex.
\end{enumerate}
The snowflake containing $v_4$ is called the {\em related snowflake} of $f$.
 We  further call $x$ a {\em 2-nice vertex} of $f$ if $f$ satisfies   (1) above, a {\em 1-nice vertex} of $f$ otherwise.
\end{definition}

 It is easy to see that a nice 9-face is related to precisely one snowflake.
 
Consider a plane graph $G$. For $ Y \subseteq V(G)$ or $ Y \subseteq E(G)$, denote by $G[Y]$ the subgraph of $G$ induced by $Y$. For a subgraph $H$ of $G$, denote by $N_G(H)$ (shortly, $N(H)$) the set of vertices of $G-V(H)$ which has a neighbor in $H$.
For a cycle $C$,   $\intt(C)$ and $\ext(C)$ are the set of vertices in the interior and exterior of  $C$, respectively. Denote by $\intt[C]$ (resp., $\ext[C]$)   the subgraph of $G$ induced by   $\intt(C) \cup V(C)$  (resp., $\ext(C) \cup V(C)$).
A cycle $C$ is \textit{separating} if both $\intt(C)$ and $\ext(C)$ are nonempty.
A path $P$ and a vertex $v$ are {\it adjacent} if $v\notin V(P)$ and $v$ is adjacent to a vertex of $P$.
A path on $k$ 2-vertices is called a {\it $k$-string} if it is adjacent to no 2-vertices. Note that if a vertex $v$ is adjacent to a string, then it is adjacent to an end vertex of the string, as vertices in a string are 2-vertices in $G$.

\section{The proof of Theorem \ref{thm_main_extension-signed}}    \label{sec_proof}

To see that Theorem \ref{thm468-DP} follows from Theorem \ref{thm_main_extension-signed}, 
take any $S_3$-signed graph $(G, \sigma)$.
If $G$ has no triangles, then it has girth at least 5 and is known to be DP-$3$-colorable \cite{Dvorak-Postle-2018}. We may next assume that $G$ has a triangle $T$. Any 3-coloring of $(T,\sigma)$ can be extended to both $(\ext[T],\sigma)$ and $(\intt[T],\sigma)$ by Theorem \ref{thm_main_extension-signed}, which together result in a 3-coloring of $(G,\sigma)$.

The remainder of this paper is devoted to the proof of Theorem \ref{thm_main_extension-signed}.

Assume to the contrary that Theorem \ref{thm_main_extension-signed} is false.
Let $(G,\sigma)$ be a counterexample   with minimum $|V(G)|$. Thus the infinite face $f_0$ is a $12^-$-face, and there exists a 3-coloring $\phi_0$ of $(G[V(f_0)],\sigma)$ that cannot extend to $(G,\sigma)$.

Denote by $D$ the boundary of $f_0$.
By the minimality of $(G,\sigma)$, $D$ has no chords.

\subsection{Reducible configurations}

\begin{lemma} \label{lem_separating-cycle}
	$G$ has no separating $12^-$-cycles.
\end{lemma}

\begin{proof}
If $C$ is a  separating $12^-$-cycle of $G$, then by the minimality of $(G,\sigma)$, we can extend $\phi_0$ to $(\ext[C],\sigma)$ and then extend the resulting coloring of $C$ to $(\intt[C],\sigma)$.  
\end{proof}

\begin{lemma}\label{lem_2connected}
	$G$ is 2-connected. Consequently, the boundary of each face is a cycle.  
\end{lemma}
\begin{proof}
	Otherwise, we may assume that $G$ has a block $B$ and a cut vertex $v\in V(B)$ with $V(D) \cap V(B-v)=\emptyset$.  
	Let $H=G-(B-v).$
	By the minimality of $(G,\sigma)$, we can extend $\phi_0$ to $(H,\sigma)$.
	Let $C$ be a cycle of $B$ of  minimum length that contains $v$.
	If $|C|\leq 12$, then $C$ is a facial cycle, since $C$ has no chords by its minimality and $C$ is not separating by Lemma \ref{lem_separating-cycle}. We can extend the coloring of $v$ to a 3-coloring of $(C,\sigma)$ and further to $(B,\sigma)$ by the minimality of $(G,\sigma)$.
    If $|C| > 12$, then insert an edge with an arbitrary sign between any two consecutive neighbors (say $x$ and $y$) of $v$ in $B$. Note that $B+xy\in \mathcal{G}$. We can extend the coloring of $v$ to a 3-coloring of $([vxy],\sigma)$ and further to $(B+xy,\sigma)$. In either case, the resulting coloring of $(G,\sigma)$ is an extension of $\phi_0$, a contradiction. 
\end{proof}

Since $G\in\mathcal{G}$, the following corollary is a consequence of Lemmas \ref{lem_separating-cycle} and \ref{lem_2connected}.
\begin{corollary} \label{cor-facial}
	Every $k$-cycle of $G$ with $k\in\{3,5,7,9\}$ is facial.
\end{corollary}

\begin{lemma} \label{lem_min degree}
Every internal vertex of $G$ has degree at least 3.
\end{lemma}
\begin{proof}
 If $v$ is an internal vertex with $d(v)\leq 2$, then we can extend $\phi_0$ to $(G-v,\sigma)$, and then extend to $(G,\sigma)$ by coloring $v$ with a color not matched to the colors of its two neighbors.  
\end{proof}

\begin{lemma} \label{lem_string}
	If $f \ne f_0$ is a $k$-face of $G$ and $P$ is a $t$-string contained in the boundary of $f$, then $t < \lfloor  \frac{k-1}{2} \rfloor$.
\end{lemma}

\begin{proof}
By Lemma \ref{lem_min degree}, $P$ is contained in the boundary of $f_0$.  Assume $t \ge  \lfloor  \frac{k-1}{2} \rfloor$.  Let $G'=G-V(P)$ and $f_0'$ be the infinite face of $G'$. Then $d(f_0')=d(f_0)+k-2(t+1)\leq d(f_0)$. We can first extend $\phi_0$ to $(f_0\cup f_0',\sigma)$ and then extend the coloring of $f'_0$ to $(G',\sigma)$.
\end{proof}


\begin{lemma}\label{lem_recolor}
	Let $[uvw]$ be a 3-face such that $d(u)=d(v)=3$. We may assume that edges of $u'uwvv'$ are all straight, where $u'$ and $v'$ are the other neighbor  of $u$ and $v$, respectively. Let $\phi$ be a 3-coloring of $u',v'$.
	\begin{enumerate}[(1)]
		\setlength{\itemsep}{0pt}
		\item If $uv$ is straight and $\phi(u')\neq \phi(v')$, then for any $c \in [3]$,   $\phi$ can be extended to $[uvw]$ with $\phi(w)=c$. 
		\item If $uv$ is not straight, then for   at least two colors $c \in [3]$,  $\phi$ can be  extended to $[uvw]$ with $\phi(w)=c$. 
	\end{enumerate}
\end{lemma}
The proof of Lemma \ref{lem_recolor} is a straightforward verification and hence omitted. 

   Lemma \ref{lem_recolor} (2)    says that if $uv$ is not straight, then any $3$-coloring of $u',v'$ forbids at most one color for $w$. 
   Thus pre-coloring $u'$ and $v'$  has the same effect as pre-coloring one  neighbor of $w$. This property is used in the proof of Lemma \ref{lem_I-segment} below and also in some later arguments.

\begin{lemma}\label{lem_I-segment}
    Let $u_0$ be the port of $I_k$. For any 3-coloring $\phi$ of $N(I_k) \setminus N(u_0)$, there exist at least two colors $c \in [3]$ such that $\phi$ can be  extended to $I_k$  so that $\phi(u_0)=c$.
\end{lemma}
\begin{proof}
    If $k=1$, then this is Lemma \ref{lem_recolor} (2).
    Assume $k \ge 2$ and the lemma holds for $I_{k-1}$. Let the vertices of $I_k$ be labelled as in Figure \ref{fig-IJ-segment}. Let $w'_k$ and $u'_k$ be the other neighbor of $w_k$ and $u_k$, respectively. Apply Lemma \ref{lem_recolor} to $G[\{u_{k-1},u_k,w_k,w'_k,u'_k\}]$, we conclude that there are two colors $c' \in [3]$ such that $\phi$ can be extended to $u_{k-1},u_k,w_k$ so that $\phi(u_{k-1})=c'$. 
    Thus $u_{k-1}$ can be treated as a 3-vertex with one pre-colored neighbor. By induction hypothesis, there are two colors $c \in [3]$ such that $\phi$ can be  extended to    $I_k$  so that $\phi(u_0)=c$.
\end{proof}

  \begin{figure}[ht]
	\centering
	\includegraphics[width=2.5cm]{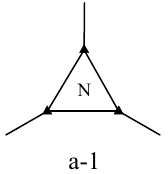}
	\caption{Configuration a-1}\label{fig-proof-a}
\end{figure}

\begin{lemma}
    \label{lem_type-a}
	If $I_k$ is contained in $G$, then the port vertex $u_0$ has degree at least $4$. In particular, $G$ has no Configuration a-1, see Figure \ref{fig-proof-a}.
\end{lemma}

\begin{proof}
	Assume to the contrary that $I_k \subseteq G$ 
 and $d_G(u_0) =3$.   
	By the minimality of $(G,\sigma)$, $\phi_0$ can be extended to $(G-I_k,\sigma)$, say the resulting coloring $\phi$.  Let $u'_0$ be the neighbor of $u_0$ in $G-I_k$. 
	By Lemma \ref{lem_I-segment}, there are at least two colors $c \in [3]$ such that $\phi(G-I_k-u'_0)$ can be extended to $I_k$ so that $\phi(u_0)=c$.
 The coloring of $u'_0$ forbids one color for $u_0$. Hence, there is at least one  color $c \in [3]$ such that  $\phi$ can be extended to  $I_k$ so that $\phi(u_0)=c$.
\end{proof}

 \begin{figure}[ht]
	\centering
	\includegraphics[width=16cm]{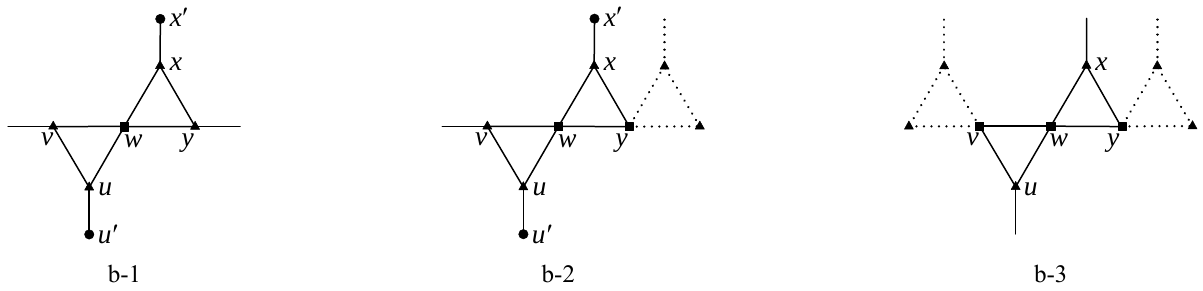}
	\caption{Configurations b-1, b-2, b-3}\label{fig-proof-b}
\end{figure}

\begin{lemma}\label{lem_type-b}
	$G$ has none of Configurations  b-1, b-2, b-3, see Figure \ref{fig-proof-b}.
\end{lemma}

\begin{proof}
	 (1) Suppose to the contrary that $G$ has Configuration b-1, say $H$.
	By Remark \ref{remark1}, we may assume  that  edges of the path $u'uwxx'$ are all straight.
	Remove  $V(H)$ from $(G,\sigma)$ and identify $u'$ with $x'$. Denote by $(G',\sigma')$ the resulting $S_3$-signed graph.
	
 We shall show that $(G',\sigma')\in \mathcal{G}$ and $\phi_0$ is still proper in $G'$.
	
	We claim that the operation creates no $8^-$-cycles and consequently, $(G',\sigma')\in \mathcal{G}$. Otherwise,  this new cycle corresponds to an $8^-$-path in $G$ connecting $u'$ and $x'$, which together with $u'uwxx'$ forms a $12^-$-cycle, say $C$. As one of $v$ and $y$ lies inside $C$ and the other lies outside, $C$ is a separating $12^-$-cycle of $G$, contradicting Lemma \ref{lem_separating-cycle}. 

	Also the operation does not identify an external vertex with another vertex which is either external or adjacent to an external vertex.
	Otherwise, the operation creates a cycle $C$ formed by a path of $D$ and possibly one more edge with $|C|\leq \frac{|D|}{2}+1$. Since $|D|\leq 12$, we have $|C|\leq 7$, contradicting the conclusion above that the operation creates no $8^-$-cycles.   
	
	Therefore, $\phi_0$ is still proper in $G'$, and by the minimality of $(G,\sigma)$, $\phi_0$ can be extended to $(G',\sigma')$ and further to $(G,\sigma)$ as follows: color $w$ the same as $u'$ and consequently, $v$ and $u$ (as well as $y$ and $x$) can be properly colored in turn.
	
	(2) Suppose to the contrary that $G$ has Configuration b-2, say $H$. Note that $H$ is an $I_k$-extension  of Configuration b-1 at $y$. So, we can apply a  similar proof as for (1), deriving a contradiction.
    Note that the graph operation here removes $V(I_k)$ instead of $y$.
    For the coloring $\phi$ extended from $(G',\sigma')$ to $(G,\sigma)$:  
	by Lemma \ref{lem_I-segment}, there are two colors $c \in [3]$ such that $\phi$ can be extended to $I_k$  so that $\phi(y)=c$. 
        This has the same effect as $y$ is a 3-vertex and has a pre-colored neighbor.
	Hence, following the proof for (1), color $w$ the same as $u'$ and consequently, $v$ and $u$ (as well as $y$ and $x$) can be properly colored in turn.
	
	(3) Suppose to the contrary that $G$ has Configuration b-3, say $H$. Note that $H$ is an $I_k$-extension of Configuration b-1 at both $y$ and $v$. We can also apply similar argument as for (1), deriving a contradiction. Note that $v$ is also the port of an $I_k$, and so it will be treated the same as $y$ for both the graph operation and the coloring extension.
 \end{proof}

  \begin{figure}[ht]
	\centering
	\includegraphics[width=10cm]{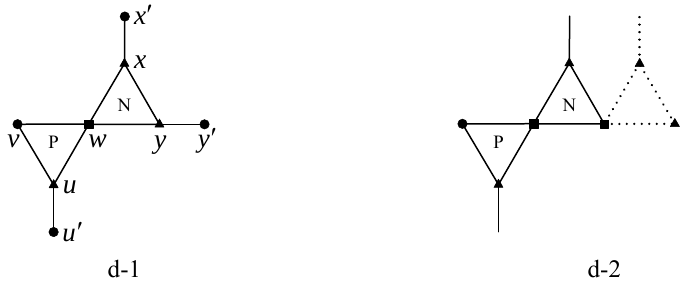}
	\caption{Configurations d-1 and d-2}\label{fig-proof-d}
\end{figure}

\begin{lemma}\label{lem_type-d}
	$G$ has none of Configurations d-1 and d-2, see Figure \ref{fig-proof-d}.
\end{lemma}

\begin{proof}
	(1) Suppose to the contrary that $G$ has Configuration d-1, say $H$.
	By Remark \ref{remark1}, we may assume edges of $x'xwuu'$ and $uvw$ are all straight. 	
	Remove $u,w,x,y$ from $(G,\sigma)$ and identify $u'$ with $x'$, obtaining a new $S_3$-signed graph $(G',\sigma')$.

We claim that the operation creates no $8^-$-cycles and consequently, $(G',\sigma')\in \mathcal{G}$. Otherwise,  this new cycle corresponds to an $8^-$-path in $G$, which together with $u'uwxx'$ forms a $12^-$-cycle, say $C$. Then  either   one of $v$ and $y$ lies inside $C$ and the other lies outside or $v\in V(C)$. In the former case, $C$ is a separating $12^-$-cycle of $G$, contradicting Lemma \ref{lem_separating-cycle}. In the latter case, $uv$ and $vw$ divide $C$ into a 3-cycle and two $9^+$-cycles, contradicting the fact that $|C|\leq 12$.
 
	Similarly as the proof of Lemma \ref{lem_type-b}, we can show that $\phi_0$ is still proper in $G'$.  
	Hence, by the minimality of $(G,\sigma)$, $\phi_0$ can be extended to $(G',\sigma')$ and further to $(G,\sigma)$ as follows: let $\alpha$ and $\beta$ be the colors of $u'$ and 
$v$, respectively. 
     If $\alpha\neq \beta$, then color $w$ with $\alpha$, which obviously extends to $u, y ,x$, since $x', w, u'$ are of the same color.
     Assume  $\alpha=\beta$.  
     Note that $w$ has a pre-colored neighbor $v$. By Lemma \ref{lem_recolor}, there exists an available color for $w$ such that the resulting coloring is extendable to $x$ and $y$. As $v$ and $u'$ are of the same color, $u$ can be properly colored.
	
	(2) Note that Configuration d-2 is an $I_k$-extension of Configuration d-1 at $y$. By similar argument as for (1), we can show that $G$ has no Configuration d-2.
\end{proof}

  \begin{figure}[ht]
	\centering
	\includegraphics[width=16cm]{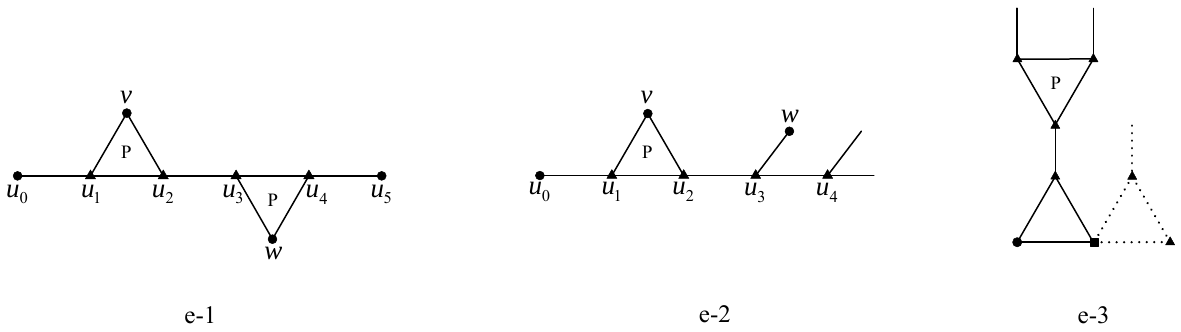}
	\caption{Configurations e-1, e-2, e-3}\label{fig-proof-e}
\end{figure}

\begin{lemma} \label{lem_type-e}
	$G$ contains none  of Configurations  e-1, e-2, e-3, see Figure \ref{fig-proof-e}. In particular, the outer neighbor of a $3_{\Delta^+}$-vertex is not a $3_{\Delta^+}$-vertex, and the outer neighbor of a bad vertex is neither a $3_{\Delta^+}$-vertex nor a $3_{\Delta^-}$-vertex.
\end{lemma}

\begin{proof}
 
	(1) Suppose to the contrary that $G$ has Configuration e-1. 
 By Remark \ref{remark1}, we may assume edges incident with $u_1$, $u_2$, $u_3$, or $u_4$ are all straight.
 Remove $u_1,u_2, u_3,u_4$ from $(G,\sigma)$ and add a straight edge between $u_0$ and $u_5$, obtaining a new $S_3$-signed graph $(G',\sigma')$.
	
	We shall show that the operation creates no $8^-$-cycles and consequently, $(G',\sigma')\in \mathcal{G}$.
	If not, a $12^-$-cycle $C$ of $G$ can be obtained from this $8^-$-cycle by substituting $u_0u_1\cdots u_5$ for $u_0u_5$. If $v\in V(C)$, then $vu_1$ and $vu_2$ divide $C$ into a 3-cycle and two $9^+$-cycles, contradicting the fact that $|C|\leq 12$.
	Thus $v\notin V(C)$ and similarly, $w\notin V(C)$. Now $C$ is a separating $12^-$-cycle, contradicting Lemma \ref{lem_separating-cycle}.  
	
	We shall show that the operation does not add an edge between two external vertices and consequently, $\phi_0$ is still proper in $G'$.
	If not,	the operation creates a cycle $C$ formed by a path of $D$ and $u_0u_5$ with $|C|\leq \frac{|D|}{2}+1\leq 7$, a contradiction.

	By the minimality of $(G,\sigma)$, $\phi_0$ can be extended to $(G',\sigma')$ and further to $(G,\sigma)$ as follows: Let $\phi$ be the resulting coloring of $(G',\sigma')$. If $\phi(u_0)=\phi(v)$ or $\phi(u_5)= \phi(w)$, then obviously $\phi$ can be extended to $(G,\sigma)$ in the order $u_4\rightarrow u_3\rightarrow u_2\rightarrow u_1$ or $u_1\rightarrow u_2\rightarrow u_3\rightarrow u_4$, respectively; otherwise, color $u_2$ same as $u_0$, and $u_3$ same as $u_5$ and consequently, we can properly color $u_1$ and $u_4$.

	(2) Suppose to the contrary that $G$ has Configuration e-2. 
        By Remark \ref{remark1}, we may assume edges incident with $u_1$, $u_2$, or $u_3$ are all straight.
	Remove $u_1,u_2,u_3,u_4$ from $(G,\sigma)$ and identify $u_0$ with $w$. We thereby obtain a new $S_3$-signed graph $(G',\sigma')$.
	Similarly as the proof of Lemma \ref{lem_type-d}, we can show that $(G',\sigma')\in \mathcal{G}$ and $\phi_0$ is still proper in $G'$.
	Therefore, by the minimality of $(G,\sigma)$, $\phi_0$ can be extended to $(G',\sigma')$ and further to $(G,\sigma)$ as follows: properly color $u_4$ and $u_3$ in turn. Since $u_0$ and $u_3$ receive different colors, the resulting coloring can be extended to $u_1$ and $u_2$ by Lemma \ref{lem_recolor}. 

    (3) Note that Configuration e-3 is an $I_k$-extension of Configuration e-2 at $u_4$ for which $u_4w\in E(G)$ and $v$ is an internal 3-vertex in $G$. 
    By similar argument as for (2), we can show that $G$ has no Configuration e-3.
\end{proof}

\begin{lemma}\label{lem_J-2-segment}
	Let $H=J_2\subseteq G$. Then the unique 4-vertex of $H$ which is not a port is  incident with two 9-faces.
\end{lemma}

\begin{proof}
	 Let $f$ and $g$ be the two faces (other than triangles) incident with $u_2$. Assume to the contrary that at least one of $f,g$ is not a $9$-face. Then $d(f)+d(g)\geq 19$.
     We distinguish two cases, see Figure \ref{fig-proof-J2}.

 \begin{figure}[ht]
	\centering
	\includegraphics[width=13cm]{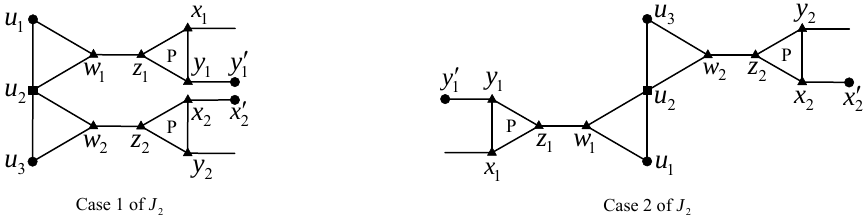}
	\caption{Configuration $J_2$ in two cases for the proof of Lemma \ref{lem_J-2-segment}}\label{fig-proof-J2}
\end{figure}

	{\bf Case 1:} $[u_1w_1u_2]$ and $[u_2w_2u_3]$ locate on the same side of $u_1u_2u_3$.

By Remark \ref{remark1}, we may assume  that  edges of paths $u_1w_1z_1y_1y_1'$ and $u_3w_2z_2x_2x_2'$ are all straight.
	Remove all the vertices of $H$ except $u_1$ and $u_3$, and identify $u_1$ with $y_1'$, and $u_3$ with $x_2'$. We thereby obtain from $(G,\sigma)$ a new $S_3$-signed graph, say $(G',\sigma')$.
	
	We claim that the operation creates no $8^-$-cycles and consequently, $(G',\sigma')\in \mathcal{G}$. Otherwise,  this new cycle contains  either one path (say $P$) between $u_1$ and $y'_1$ or between $u_3$ and $x'_2$, or two paths such that one (say $P_1$) between $u_1$ and $u_3$ and the other (say $P_2$) between $y_1'$ and $x_2'$ in $G$. 
	In the former case, $P$ together with $u_1w_1z_1y_1y_1'$ or $u_3w_2z_2x_2x_2'$ forms a separating $12^-$-cycle of $G$, contradicting Lemma \ref{lem_separating-cycle}. 
	In the latter case, it follows that $d(f)+d(g)=|P_1|+|P_2|+10 \leq 18$, contradicting the assumption that $d(f)+d(g)\geq 19$.
	
	Similarly as the proof of Lemma \ref{lem_type-b}, we can show that $\phi_0$ is still proper in $G'$.
	
	Therefore, by the minimality of $(G,\sigma)$, $\phi_0$ can be extended to $(G',\sigma')$ and further to $(G,\sigma)$ as follows: properly color $u_2$ and then $w_1$ and $w_2$. Properly color $x_1$ and $y_2$. By Lemma \ref{lem_recolor}, the resulting coloring is extendable to $y_1$ and $z_1$, and to $x_2$ and $z_2$ as well.
	
	{\bf Case 2:}   $[u_1w_1u_2]$ and $[u_2w_2u_3]$ locate on distinct sides of $u_1u_2u_3$. 
 
 Note that $x_2,y_2,z_2$ locate in anti-clockwise order around the incident 3-face while $x_1,y_1,z_1$ do in clockwise order.
	By the same argument as for Case 1, a contradiction can be derived.
\end{proof}

\begin{lemma}\label{lem_J-segment}
	Let $H=J_k\subseteq G$ with $k\geq 4$. Then among all the 9-faces containing a 4-vertex of $H$ which is not a port,  there exist at least $k-2$ nice 9-faces.
\end{lemma}

\begin{proof}
	Denote by $P=u_1u_2\ldots u_{k+1}$ the unique path of $H$ with $u_1$ and $u_{k+1}$ being ports of $H$ and $u_2, u_3, \ldots, u_k$ being $4_{\bowtie}$-vertices. Let $T_i=[u_iw_iu_{i+1}]$ for $i\in\{1,2,\ldots,k\}$. For any $1\leq i < j \leq k$ such that $T_i$ and $T_j$ locate on one side of $P$ and $T_{i+1},\ldots,T_{j-1}$
 locate on the other side, let the face $f_{ij}=[y_iz_iw_iu_{i+1}u_{i+2}\ldots u_jw_jz_jx_jo_1o_2\ldots]$. See Figure \ref{fig-proof-J-nice} for an example of $J_4$ and $f_{ij}$.
 
\begin{figure}[ht]
	\centering
	\includegraphics[width=4.5cm]{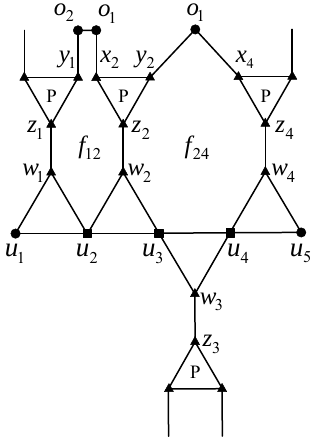}
	\caption{An example of $J_4$ for the proof of Lemma \ref{lem_J-segment}}\label{fig-proof-J-nice}
\end{figure}	

	We claim that each $f_{ij}$ is a nice 9-face and $1\leq j-i\leq 2$. By Lemma \ref{lem_J-2-segment}, $d(f_{ij})=9$. This implies that  $1\leq j-i\leq 3$.
	Since $G$ has no Configuration e-2, we can deduce that: (1) $y_ix_j\notin E(G)$, i.e., $j-i\neq 3$;
	(2) if $j-i=2$, then $o_1$ is not an internal 3-vertex; (3) if $j-i=1$, then at least one of $o_1$ and $o_2$ is not an internal 3-vertex. By definition, $f_{ij}$ is a nice 9-face.

	Let $F$ be the set of all such faces $f_{ij}$. By the claim above, $|F|\geq k-2$.
\end{proof}

\begin{figure}
	\centering
	\includegraphics[width=15cm]{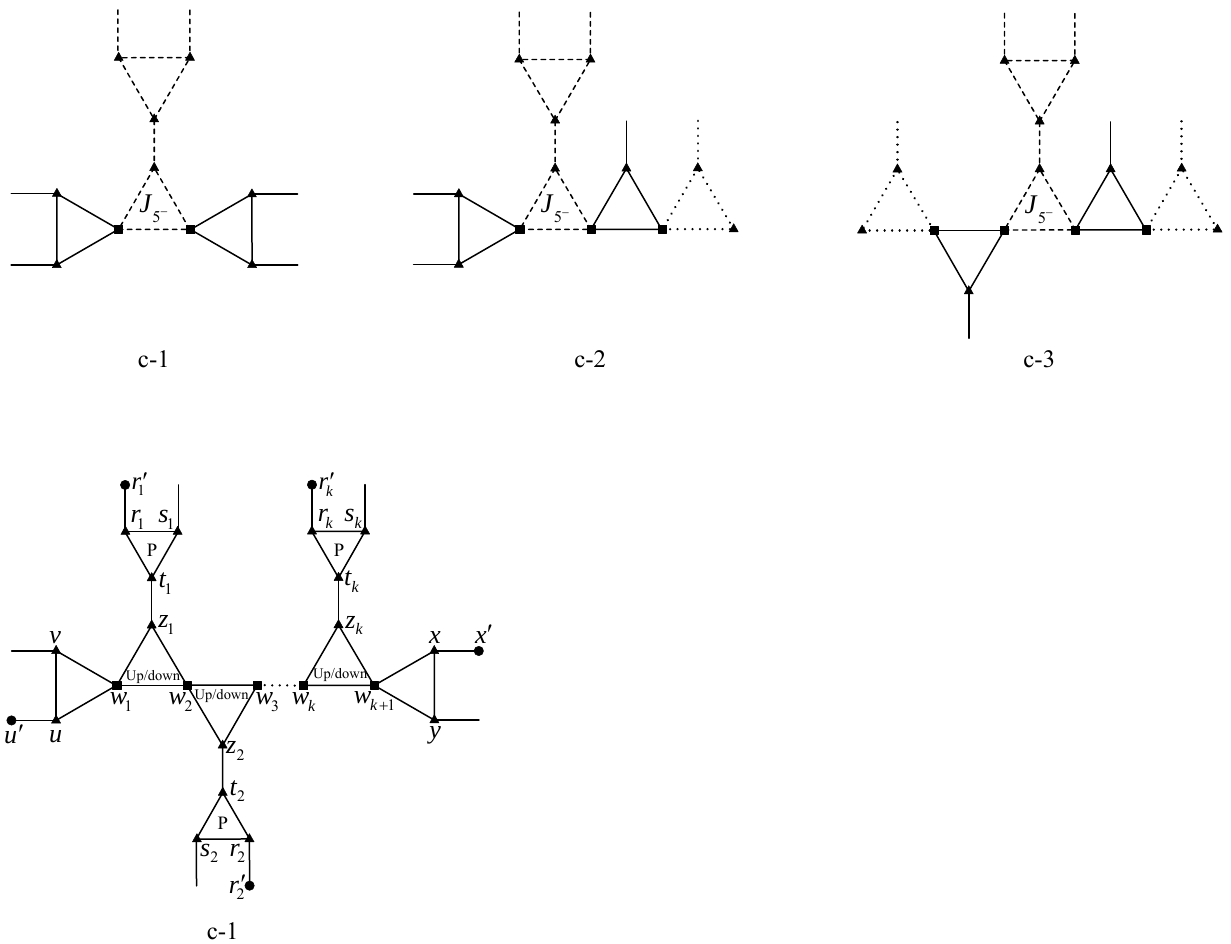}
	\caption{Configurations c-1, c-2 and c-3, where $J_{5^-}$ indicates a $J_{k}$-extension with any $k\leq 5$}\label{fig-c}
\end{figure}

\begin{lemma}\label{lem_type-c}
	$G$ has none  of Configurations c-1, c-2, c-3, see Figure \ref{fig-c}.
\end{lemma}

\begin{proof}	
	(1) Suppose to the contrary that $G$ has Configuration c-1, say $H$.
	Note that for each $i\in\{1,2,\ldots,k\}$, the vertices $r_i,s_i,t_i$ are labelled in clockwise order around their incident triangle no matter on which side of the path $w_1w_2\ldots w_{k+1}$ this triangle is located.
	By Remark \ref{remark1}, we may assume that edges of $u'uw_1$, $z_1t_1r_1r_1'$, $z_2t_2r_2r_2'$, $\ldots$, $z_kt_kr_kr_k'$, $w_{k+1}xx'$ are all straight.
	Let $Q=[w_1z_1w_2]\cup [w_2z_2w_3]\cup \cdots \cup [w_kz_kw_{k+1}]$.
	Remove $V(H)\setminus V(Q)$ from $(G,\sigma)$, identify $u'$ with $w_1$ into a new vertex $w^*_1$, and $w_{k+1}$ with $x'$ into a new vertex $w^*_{k+1}$, and insert a new straight edge   $e_i$ between $r_i'$ and $z_i$ for each $i\in\{1,2,\ldots,k\}$. Denote by $(G',\sigma')$ the resulting $S_3$-signed graph.
	We shall show that $(G',\sigma')\in \mathcal{G}$ and $\phi_0$ is still proper in $G'$.
	
	We claim that the operation creates no $8^-$-cycles and consequently, $(G',\sigma')\in \mathcal{G}$.
	Otherwise, denote by $C'$ this new $8^-$-cycle. A {\em segment} of $C'$ is a   subpath $P'$ of $C'$ such that each of its  two ends  is either a new vertex $w^*_1$ or $w^*_{k+1}$, or a new edge   $e_j = z_jr'_j$, and whose interior vertices are in $Q$.  Let  $P'_1,P'_2,\ldots,P'_l$ be the segments of $C'$. 
	For $1\leq i \leq l$, construct a path $P_i$ from $P'_i$  as follows: if $e_j= z_jr'_j$ is an edge of $P'_i$, then   replace $e_j$ by $z_jt_jr_jr_j'$; if $w^*_1 $ (resp., $w^*_{k+1}$) is an end vertex of $P'_i$, then replace $w^*_1$ by $w_1uu'$ (resp., replace $w^*_{k+1}$    by   $w_{k+1}xx'$). 
	Denote by $C$ the cycle obtained from $C'$ by replacing $P'_i$ by $P_i$ for   $i=1,2,\ldots, l$.
	
    Note that if the two ends of $P'_i$ are new edges, then $P'_i$ contains at least two edges of $Q$, and hence $|P'_i| \ge  4$; and $|P'_i| \ge 2$ if one end of $P'_i$ is a new vertex and the other end is a new edge.  As there are only two new vertices, $ 8 \ge |C'| \ge \sum_{i=1}^{l}|P'_i|\geq 4l-4$ and hence $1\leq l\leq 3$. 
   
   It follows from the construction that $|P_i|=|P'_i|+4$ and $|C|=|C'|+4l$.

	If $l=1$, then   $C$ is a separating $12^-$-cycle of $G$, contradicting Lemma \ref{lem_separating-cycle}.
	If $l\in \{2,3\}$, then 
	  $C-\bigcup_{i=1}^{l}E(P_i)\cap E(Q)$ consists of $l$ paths $R_1,R_2,\ldots, R_l$ in $G-Q$, each    $R_i$ together with the shortest path $R'_i$ of $Q$ connecting the two end vertices of $R_i$ forms a cycle $C_i$ of $G$. 
        Note that each $R'_i$ has length at most $k \le 5$.
        Moreover, each path $E(P_i)\cap E(Q)$ has length at least $2$, except that there are at most two such paths that  has one end vertex in $\{w_1, w_{k+1}\}$, and has length at least 1. So 
        $8+4l \ge |C'|+4l=|C| \ge \sum_{i=1}^l |R_i| + 2l-2$ and hence, 
        $\sum_{i=1}^l|C_i| = \sum_{i=1}^l (|R_i|+|R'_i|) \le \sum_{i=1}^l|R_i| + 5l \le (10+2l) + 5l =10+7l$. So one of the $C_i$ is a separating $12^-$-cycle of $G$, a contradiction.

	Note that all the vertices of $Q$ are internal. The operation makes $\phi_0$ still proper in $G'$.
	
	By the minimality of $(G,\sigma)$, $\phi_0$ can be extended to $(G',\sigma')$ and further to $(G,\sigma)$ as follows: since $u'$ and $w_1$ receive the same color, we can properly color $v$ and $u$ in turn, and so do $y$ and $x$. For $1\leq i \leq k$, properly color $s_i$.  Since $z_i$ and $r_i'$ receive distinct colors, the resulting coloring can be extended to $r_i$ and $t_i$ by Lemma \ref{lem_recolor}.

    (2) By definition, Configuration c-2 (resp., c-3) is an $I_k$-extension of Configuration c-1 at $y$ (resp., at $y$ and $v$). So by a similar proof as for (1), we can show that $G$ has neither c-2 nor c-3.
\end{proof}

\begin{figure}
        \centering
	\includegraphics[width=16cm]{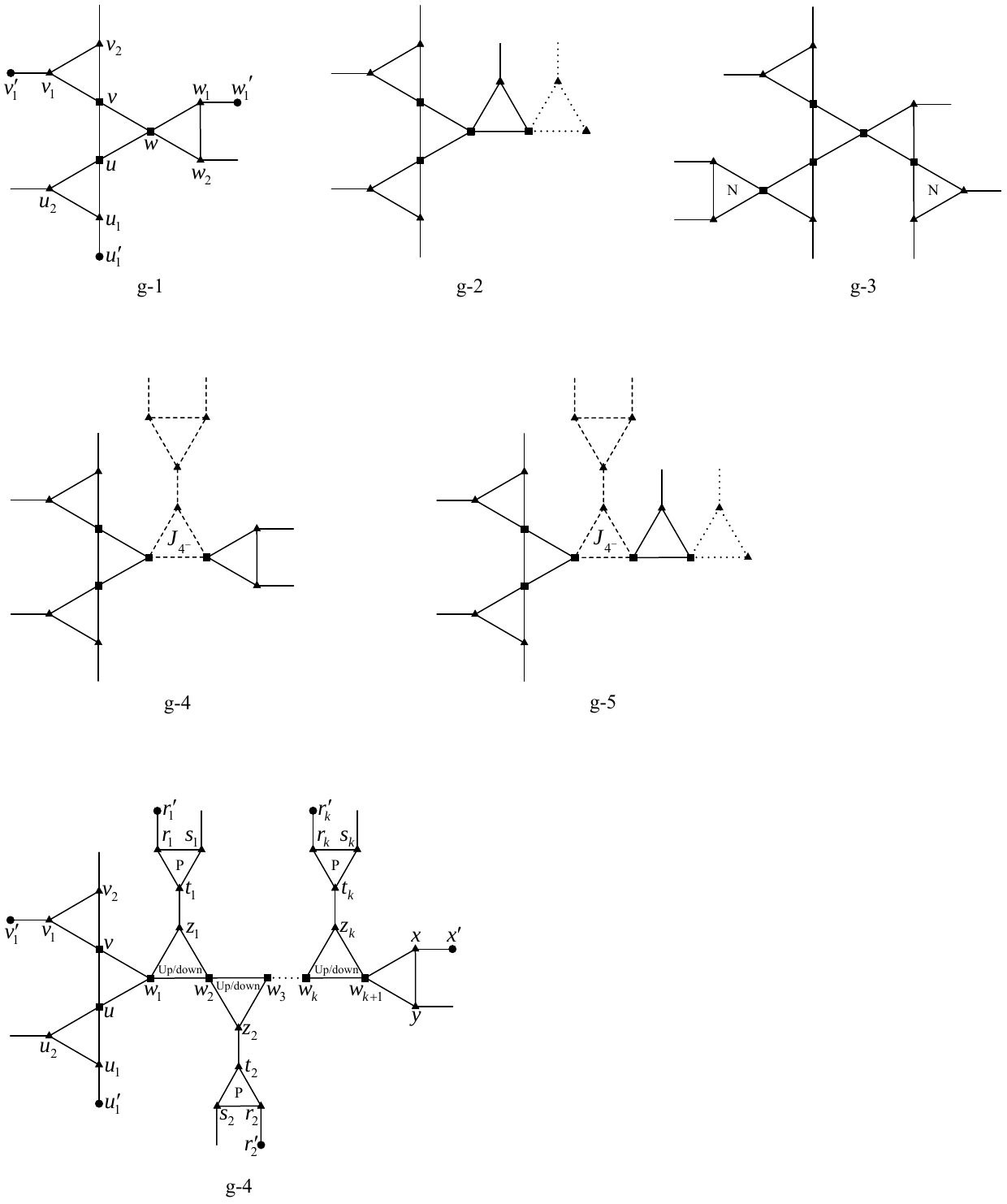}
        \caption{Configurations g-1, g-2, g-3, g-4, g-5, where $J_{4^-}$ indicates a $J_{k}$-extension with any $k\leq 4$}
        \label{fig-g}
        \FloatBarrier
\end{figure}

\begin{lemma}\label{lem_type-g}
	$G$ has none  of Configurations  g-1, g-2, g-3, g-4, g-5, see Figure \ref{fig-g}.
\end{lemma}

\begin{proof}
    (1) Suppose to the contrary that $G$ has Configuration g-1, say $H$.
	By Remark \ref{remark1}, we may assume edges of $u_1'u_1u, v_1'v_1v, w_1'w_1w$ are all straight.
	Remove $u_1,u_2,v_1,v_2,w_1,w_2$ from $(G,\sigma)$ and identify $u$ with $u_1'$, $v$ with $v_1'$, and $w$ with $w_1'$, obtaining a new $S_3$-signed graph $(G',\sigma')$. 
	
	We claim that the operation creates no $8^-$-cycles and consequently, $(G',\sigma')\in \mathcal{G}$.
	Otherwise, denote by $C$ this new $8^-$-cycle. Clearly, $C$ contains precisely one path of $[uvw]$ (that is either an edge or a 2-path). W.l.o.g., let $u$ and $v$ be ends of this path.
	So, $C$ corresponds to two paths of $G$, which together with $u_1'u_1u$ and $vv_1v_1'$ forms a separating $12^-$-cycle of $G$, contradicting Lemma \ref{lem_separating-cycle}.
	
	We claim that $\phi_0$ is still proper in $G'$. Otherwise, there exist at least two external vertices among $u'_1$, $v'_1$, and $w'_1$, w.l.o.g., say $u'_1$ and $v'_1$. So, the operation creates a cycle $C'$ formed by $uv$ and a path of $D$ with $|C'|\leq \frac{|D|}{2}+1\leq 7$, a contradiction.
	
	By the minimality of $(G,\sigma)$, $\phi_0$ can be extended to $(G',\sigma')$ and further to $(G,\sigma)$ as follows: since $u$ and $u_1'$ receive the same color, $u_2$ and $u_1$ can be properly colored in turn. So do $v_2$ and $v_1$, as well as $w_2$ and $w_1$. 
	
	(2) Note that Configuration g-2 is an $I_k$-extension of Configuration g-1 at $w_2$,
	and Configuration g-3 is an $I_1$-extension of Configuration g-1 at both $w_2$ and $u_2$. Therefore, by a similar argument as for (1), we can show that $G$ has neither g-2 nor g-3. 
	 
	 (3) Suppose to the contrary that $G$ has Configuration g-4, say $H$. Note that $H$ is a $J_k$-extension of Configuration g-1 at $w$ for some $k\leq 4$. We will follow the proof for the reducibility of both c-1 and g-1.
	 
	 By Remark \ref{remark1}, we may assume that edges of $u_1'u_1u$, $v_1'v_1v$, $z_1t_1r_1r_1'$, $z_2t_2r_2r_2'$, $\ldots$, $z_kt_kr_kr_k'$, and $w_{k+1}xx'$ are all straight.
	 Let $Q=[uvw_1]\cup [w_1z_1w_2]\cup \cdots \cup [w_kz_kw_{k+1}] \cup [w_{k+1}xy]$.
	 Remove $V(H)\setminus V(Q)$ from $(G,\sigma)$, identify $u_1'$ with $u$, $v'$ with $v$, and $x'$ with $w_{k+1}$, and insert a straight edge $e_i$ between $r_i'$ and $z_i$ for 
      each $i\in\{1,2,\ldots,k\}$. 
      Denote by $(G',\sigma')$ the resulting $S_3$-signed graph.

      By similar proof as for the reducibility of Configuration c-1, 
      we can show that $(G',\sigma')\in \mathcal{G}$ and $\phi_0$ is still proper in $G'$. 
      Note that here the distance of any two vertices of $Q$ is at most $k+1$, which is still no more than $5$, since $k\leq 4$. So, the similar proof still works.
      Finally, $\phi_0$ can be extended to $(G',\sigma')$ and further to $(G,\sigma)$ in a similar way.

	 (4) Note that Configuration g-5 is an $I_k$-extension of Configuration g-4 at $y$. By similar argument as for (3), we can show that $G$ has no g-5. 
\end{proof}

\begin{figure}
	\centering
	\includegraphics[width=15cm]{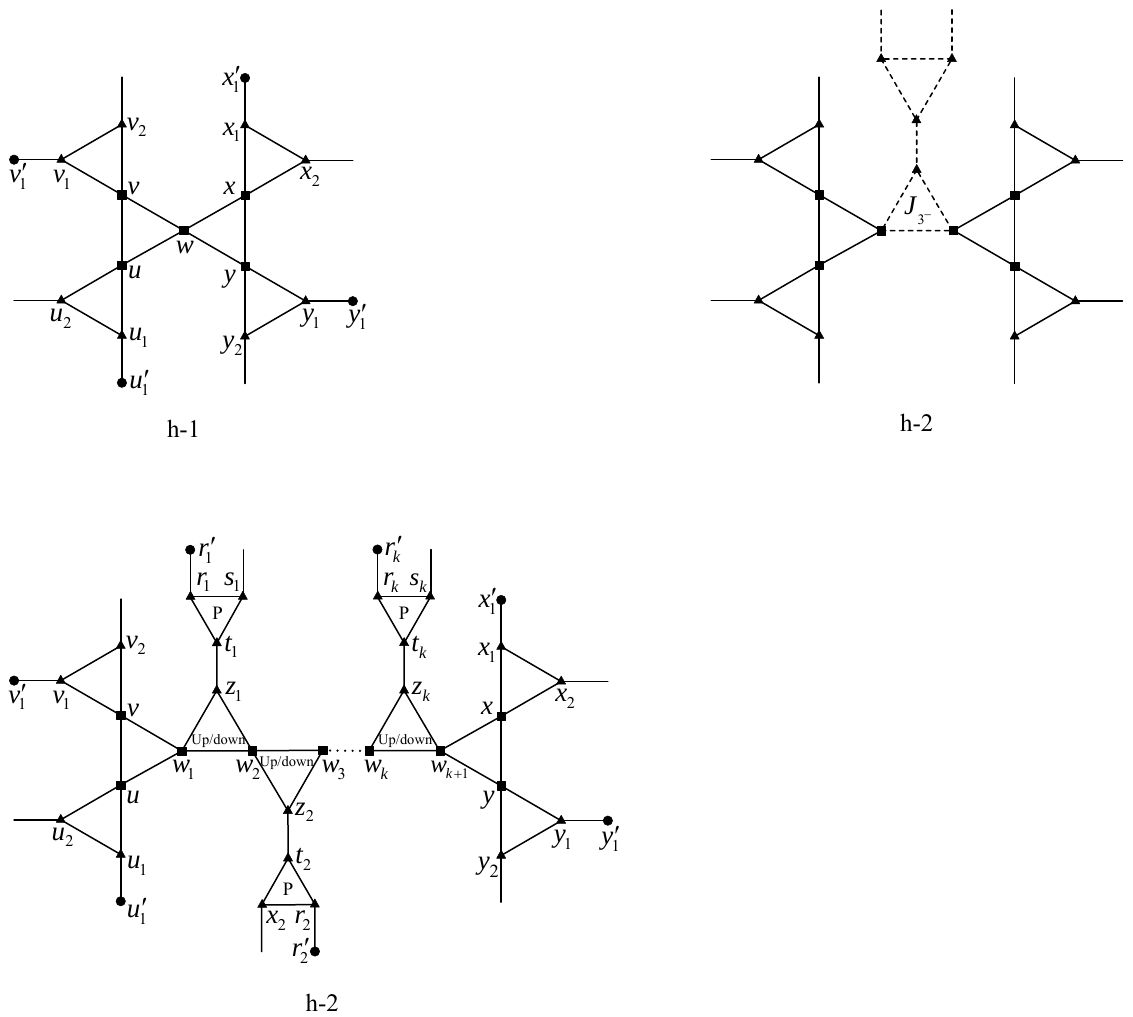}
	\caption{Configurations h-1 and h-2}\label{fig-proof-h}
\end{figure}

\begin{lemma}\label{lem_type-h}
	$G$ has none of Configurations  h-1 and h-2, see Figure \ref{fig-proof-h}.
\end{lemma}

\begin{proof}
	(1) Suppose to the contrary that $G$ has Configuration h-1, say $H$.
	By Remark \ref{remark1}, we may assume that edges of $u_1'u_1u$, $v_1'v_1v$, $x_1'x_1x$, and $y_1'y_1y$ are all straight.
	Let $Q=[uvw]\cup [wxy]$.
	Remove $V(H)\setminus V(Q)$ from $(G,\sigma)$ and identify $u$ with $u_1'$, $v$ with $v_1'$, $x$ with $x_1'$, and $y$ with $y_1'$ (resulting in four new vertices $u^*, v^*, x^*, y^*$, respectively). We thereby obtain a new $S_3$-signed graph $(G',\sigma')$.

    We claim that the operation creates no $8^-$-cycles and consequently, $(D',\sigma')\in \mathcal{G}$.
    Otherwise, denote by $C'$ this new $8^-$-cycle. 
    A {\em segment} of $C'$ is a subpath $P'$ of $C'$ such that both of its two ends are in $\{u^*, v^*, x^*, y^*\}$, and whose interior vertices are in $Q$.  Let  $P'_1,P'_2,\ldots,P'_l$ be the segments of $C'$. By the structure of $Q$, $1\leq l\leq 2$.
    For $1\leq i \leq l$, construct a path $P_i$ from $P'_i$  as follows: if $u^*$ (resp., $v^*, x^*, y^*$) is an end vertex of $P'_i$, then replace $u^*$ by $uu_1u'_1$ (resp., replace $v^*$ by $vv_1v'_1$, replace $x^*$ by $xx_1x'_1$, replace $y^*$ by $yy_1y'_1$). 
	Denote by $C$ the cycle obtained from $C'$ by replacing $P'_i$ by $P_i$ for   $i=1,2,\ldots, l$.
	Clearly, $|C|=|C'|+4l$.
	For $l=1$,  $C$ is a separating $12^-$-cycle of $G$, contradicting Lemma \ref{lem_separating-cycle}.
	For $l=2$, one of $P'_1$ and $P'_2$ connects $u^*$ with $v^*$ and the other connects $x^*$ with $y^*$. So, by replacing $P_1\cup P_2$ by $v'_1v_1vwxx_1x'_1\cup u'_1u_1uwyy_1y'_1$, we can obtain from $C$ two cycles of $G$, say $C_1$ and $C_2$, with $w$ as their only common vertex. Clearly, $|C_1|+|C_2|\leq |C|+2\leq 18$. So, at least one of $C_1$ and $C_2$ is a separating $9^-$-cycle, contradicting Lemma \ref{lem_separating-cycle}.
 
	We claim that $\phi_0$ is still proper in $G'$. Otherwise, either both $u_1'$ and $v_1'$ are external or both $x_1'$ and $y_1'$ are external.
	W.l.o.g., say the former case. Then the operation creates a cycle $C''$ formed by $u^*v^*$ and a path of $D$ with $|C''|\leq \frac{|D|}{2}+1\leq 7$, a contradiction.
	
	By the minimality of $(G,\sigma)$, $\phi_0$ can be extended to $(G',\sigma')$ and further to $(G,\sigma)$ as follows: since $u$ and $u_1'$ receive the same color, $u_2$ and $u_1$ can be properly colored in turn. So do pairs $v_2$ and $v_1$, $x_2$ and $x_1$, $y_2$ and $y_1$.
	
	(2) Suppose to the contrary that $G$ has Configuration h-2, say $H$. Note that $H$ is a $J_k$-extension of Configuration h-1 at $w$ for $k\leq 3$. 
	By Remark \ref{remark1}, we may assume that edges of $uu_1u_1'$, $vv_1v_1'$,   $z_1t_1r_1s_1'$, $z_2t_2r_2s_2'$, $\ldots$, $z_kt_kr_ks_k'$, $xx_1x_1'$, $yy_1y_1'$ are all straight.
	Let $Q=[uvw_1]\cup [w_1z_1w_2]\cup \cdots \cup [w_kz_kw_{k+1}] \cup [w_{k+1}xy]$.
	Remove $V(H)\setminus V(Q)$ from $(G,\sigma)$, identify $u'$ with $u$, $v'$ with $v$, $x'$ with $x$, and $y'$ with $y$, and insert a straight edge $e_i$ between $r_i'$ and $z_i$ for each $i\in\{1,2,\ldots,k\}$. Denote by $(G',\sigma')$ the resulting $S_3$-signed graph.

      By similar proof as for the reducibility of Configuration c-1, 
      we can show that $(G',\sigma')\in \mathcal{G}$ and $\phi_0$ is still proper in $G'$. 
      Note that here the distance between any two vertices of $Q$ is at most $k+2$, which is still no more than $5$, since $k\leq 3$. So, the similar proof still works.
      Finally, $\phi_0$ can be extended to $(G',\sigma')$ and further to $(G,\sigma)$ in a similar way. 
\end{proof}

\begin{figure}
	\centering
	\includegraphics[width=14cm]{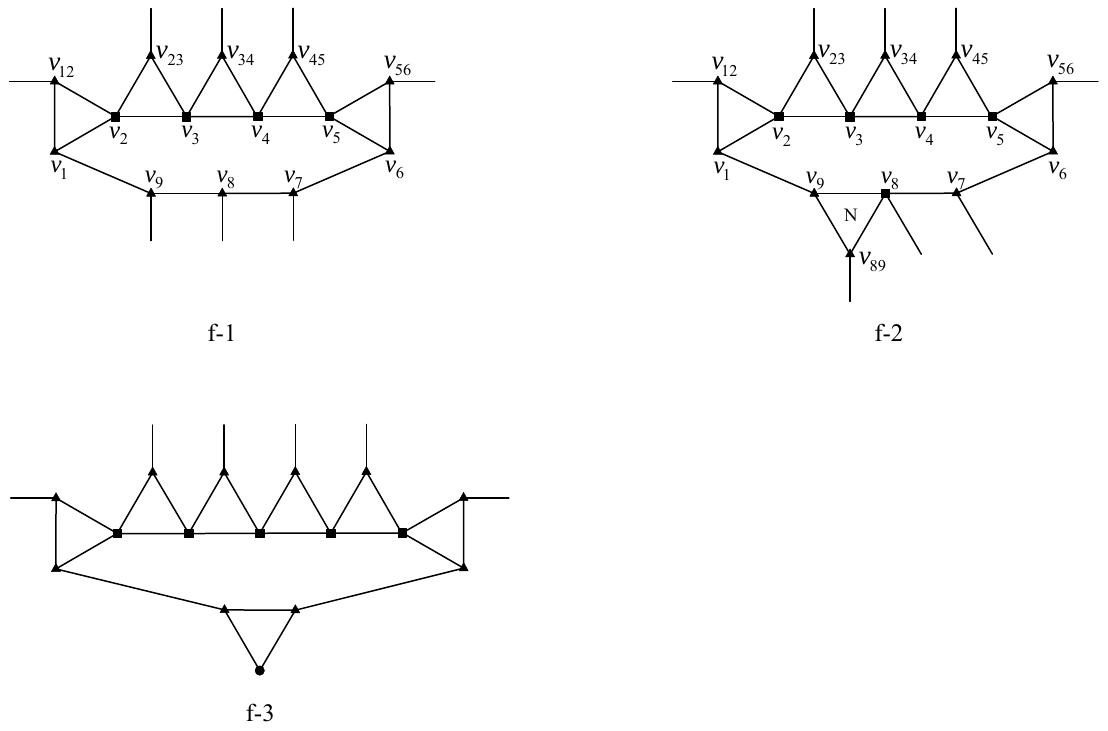}
	\caption{Configurations f-1, f-2, f-3}\label{fig-proof-f}
\end{figure}

\begin{lemma}\label{lem_type-f}
	$G$ has none of Configurations f-1, f-2, f-3, see Figure \ref{fig-proof-f}.
\end{lemma}

\begin{proof}
	(1) Suppose to the contrary that $G$ has Configuration f-1, say $H$.
	By the minimality of $(G,\sigma)$, $\phi_0$ can be extended to $(G-V(H),\sigma)$ and further to $(G,\sigma)$ as follows: 
 Note that the vertex $v_6$ has three permissible color, and $v_7$ has at least two permissible colors. Choose a color for $v_6$ so that after $v_6$ is colored, $v_7$ still has at least two permissible colors.
 Then properly color all the remaining uncolored vertices in the order 
    $v_{56}$, $v_5$, $v_{45}$, $v_4$, $v_{34}$, $v_3$, $v_{23}$, $v_2$, $v_{12}$, $v_1$, $v_9$, $v_8$, $v_7$.
	
	(2) Suppose to the contrary that $G$ has Configuration f-2, say $H$.
	By the minimality of $(G,\sigma)$, $\phi_0$ can be extended to $(G-V(H),\sigma)$ and further to $(G,\sigma)$ as follows:  
    Similarly as in (1), we can choose a color for $v_6$ so that after $v_6$ is colored, $v_7$ still has at least two permissible colors.
    Then properly color vertices  
    $v_{56}$, $v_5$, $v_{45}$, $v_4$, $v_{34}$, $v_3$, $v_{23}$, $v_2$, $v_{12}$, $v_1$ in order. Note that $v_8$ has one pre-colored neighbor. Since the triangle $[v_8v_9v_{89}]$ is negative, by Lemma \ref{lem_recolor}, the resulting coloring can be extended to $[v_8v_9v_{89}]$. Finally, properly color $v_7$.


 
	(3) By similar argument as for (1), we can also show that $G$ has no Configuration f-3.
\end{proof}

\subsection{Discharging in $G$}\label{secch}

In what follows, let $V$, $E$, and $F$ be the set of vertices, edges, and faces of $G$, respectively. 
For each $x\in V\cup F$,
the {\it initial charge} $ch(x)$ of $x$ is defined as 
\begin{equation*}\label{eq_def_initial_charge}
ch(x)=
\begin{cases}
d(x)+4, \text{~if~} x=f_0;\\
d(x)-4, \text{~otherwise}.
\end{cases}
\end{equation*}
Move charge among elements of $V\cup F$ according to the following rules:
\begin{enumerate}[$R1.$]
  \setlength{\itemsep}{0pt}
  \item $f_0$ sends to each incident vertex charge $\frac{4}{3}$.
  \label{rule-ext-face} 
  
  \item A $C$-vertex $u$ sends to each incident 3-face $f$ with $f\neq f_0$ charge $\frac{5}{9}$ if $u$ is either an external 3-vertex or an external 4-vertex incident with two 3-faces; charge $1$ otherwise.
  \label{rule-C-vertex}

  \item Every $5^+$-face $f$ with $f\neq f_0$ sends charge $\frac{d(f)-4}{d(f)}$ to each of its incident vertices.
  \label{rule-face}

  \item Let $u$ be a $3_{\Delta}$-vertex of a snowflake $S$ such that the outer neighbor $u'$ of $u$ is not a vertex of $S$. \label{rule-A1-vertex}
  \begin{enumerate}[$(1)$]
  	\setlength{\itemsep}{0pt}
    \item If $u$ is a bad vertex and $u'$ is a $3_{\Delta^{\circ}}$- or $C$-vertex, then $u'$ sends charge $\frac{2}{9}$ to $u$; 
    \item If $u$ is a $3_{\Delta^+}$-vertex but not bad, and $u'$ is a $3_{\Delta^-}$-, $3_{\Delta^{\circ}}$-, or $C$-vertex, then $u'$ sends charge $\frac{2}{27}$ to $u$;
    \item If $u$ is a $3_{\Delta^-}$-vertex, and $u'$ is a $3_{\Delta^{\circ}}$- or $C$-vertex, then $u'$ sends charge $\frac{2}{27}$ to $u$. 
    \end{enumerate}

  \item Let $f$ be a nice 9-face related to a snowflake $S$, and $u$ be a nice vertex of $f$.
  \begin{enumerate}[$(1)$]
  	\setlength{\itemsep}{0pt}
  \item If $u$ is a 1-nice vertex of $f$, then $u$ sends charge $\frac{2}{27}$ to $S$ through $f$. 
  \item If $u$ is a 2-nice vertex of $f$, then $u$ sends charge $\frac{4}{27}$ to $S$ through $f$.
  \end{enumerate}  
  \label{rule_nice face}  
  \item Let $L$ be a string on the boundary of an $11^-$-face $f$ with $f\neq f_0$. If $u$ is a vertex adjacent to $L$, then $u$ sends to each vertex of $L$ charge $\frac{12-d(f)}{6d(f)}$. 
  \label{rule-string} 
\end{enumerate}

Let $ch^*(x)$ denote the {\it final charge} of an element $x$ of $V\cup F$ after the discharging procedure.
By Euler's formula $|V|-|E|+|F|=2$ together with Handshaking Theorem $2|E|=\sum\limits_{v\in V}d(v)=\sum\limits_{f\in F}d(f)$ , we can deduce from the definition of $ch(x)$ that 
\begin{align*}
\sum\limits_{x\in V\cup F}ch(x)
&=\sum\limits_{v\in V}(d(v)-4)+\sum\limits_{f\in F}(d(f)-4)+8\\ 
&=\sum\limits_{v\in V}d(v)-4|V|+\sum\limits_{f\in F}d(f)-4|F|+8\\
&=4(|E|-|V|-|F|)+8\\
&=0.
\end{align*}
 As $\sum\limits_{x\in V\cup F}ch^*(x)=\sum\limits_{x\in V\cup F}ch(x)$, 
to complete the proof of Theorem \ref{thm_main_extension} by deriving a contradiction, it suffices to show that $\sum\limits_{x\in V\cup F}ch^*(x)>0$.

The initial charge $ch(S)$ and the final charge $ch^*(S)$ of a snowflake $S$ are defined as follows:
\begin{equation} \label{eq_1}
ch(S)=\sum_{v\in 3_{\Delta}(S)\cup 4_{\bowtie}(S)}ch(v)+\sum_{f\in T(S)}ch(f)=-|3_{\Delta}(S)|-|T(S)|.
\end{equation}
\begin{equation*} 
ch^*(S)=\sum_{v\in 3_{\Delta}(S)\cup 4_{\bowtie}(S)}ch^*(v)+\sum_{f\in T(S)}ch^*(f).
\end{equation*}

\begin{claim} \label{claim_snow}
$ch^*(S)\geq0$ for each snowflake $S$ of $G$.
\end{claim}

\begin{proof}
First assume that $S$ is a positive $(3,3,3)$-face. We can calculate from Formula \ref{eq_1} that $ch(S)=-4$. For each vertex $u$ of $S$, let $u'$ be the outer neighbor of $u$.
By Lemma \ref{lem_type-e}, $u'$ is not a $3_{\Delta^+}$-vertex or a $3_{\Delta^-}$-vertex. Clearly, $u'$ is neither a $4_{\bowtie}$-vertex nor a 2-vertex. So, $u'$ is a $3_{\Delta^{\circ}}$-vertex or a $C$-vertex.
By $R\ref{rule-A1-vertex}(1)$, $u$ receives charge $\frac{2}{9}$ from $u'$. Moreover, since $G\in \mathcal{G}$, $u$ is incident with two $9^+$-faces. By $R\ref{rule-face}$, $u$ receives charge at least $\frac{5}{9}$ from each of them. Therefore, $ch^*(S)\geq ch(S)+\frac{2}{9}\times 3+\frac{5}{9}\times 6=0.$

Assume  $S$ is not a positive $(3,3,3)$-face.
By Lemma \ref{lem_string}, $S$ contains no 2-vertices. By $R\ref{rule-face}$, each vertex of $3_{\Delta}(S)\cup 4_{\bowtie}(S)$ receives a total charge at least $\frac{5}{9}\times 2$ from its two incident $9^+$-faces.
For each $v\in 3_{\Delta}(S)$, let $v'$ be the outer neighbor of $v$.
If $v$ is a $3_{\Delta^+}$-vertex, then $v'$ is not by Lemma \ref{lem_type-e}. So, $v'$ is a $3_{\Delta^-}$-vertex  or a $3_{\Delta^{\circ}}$-vertex  or a $C$-vertex. By $R\ref{rule-A1-vertex}(2)$, $v$ receives charge $\frac{2}{27}$ from $v'$. 
If $v$ is a $3_{\Delta^-}$-vertex, then $v'$ might be a $3_{\Delta^+}$-vertex, for which case $v$ sends charge $\frac{2}{27}$ to $v'$.
If $v$ is a $3_{\Delta^{\circ}}$-vertex, then by $R\ref{rule-A1-vertex}$, $v$ sends to $v'$ charge $\frac{2}{9}$ if $v'$ is bad, and charge at most $\frac{2}{27}$  otherwise. 
Finally, the 3-faces of $S$ receive a total charge $\frac{5}{9}t_1(S)+t_2(S)$ from incident $C$-vertices by $R\ref{rule-C-vertex}$.
Therefore, 
\begin{equation}\label{eq_2}
\begin{split}
ch^*(S)\geq~
&ch(S)+\frac{5}{9}\times 2 \times (|3_{\Delta}(S)|+|4_{\bowtie}(S)|)+\frac{2}{27}\times |3_{\Delta^+}(S)|-\frac{2}{27}\times |3_{\Delta^-}(S)| \\
&-\frac{2}{9}\times (|3_{\Delta^{\circ}}(S)|-|3_{\Delta^{\star}}(S)|)-\frac{2}{27}\times |3_{\Delta^{\star}}(S)|+ \frac{5}{9}t_1(S)+t_2(S).    
\end{split}    
\end{equation}

Let $H$ be a graph whose vertex set $V(H)=T(S)$ and edge set $E(H)$ is given by: for any $f,g\in V(H)$, $fg\in E(H)$ if and only if $f$ and $g$ are intersecting at a $4_{\bowtie}$-vertex in $G$. Clearly, $H$ is a connected subcubic plane graph.

Notice that the number of 3-faces of $S$ containing $v$ is one if $v\in 3_{\Delta}(S)$, two if  $v\in 4_{\bowtie}(S)$, and $t_i(S,v)$ if $v\in C_i(S)$ for $i\in \{1,2\}$. Hence,
\begin{equation}\label{eq_3}
3|T(S)|= |3_{\Delta}(S)|+2|4_{\bowtie}(S)|+t_1(S)+t_2(S).
\end{equation}
Combining Formulas \ref{eq_1}, \ref{eq_2}, \ref{eq_3} gives
\begin{equation}\label{eq_3-3}
ch^*(S)\geq -\frac{4}{27}|3_{\Delta^+}(S)|-\frac{8}{27}|3_{\Delta^-}(S)|-\frac{4}{9}|3_{\Delta^{\circ}}(S)|+\frac{4}{27}|3_{\Delta^{\star}}(S)|+\frac{4}{9}|4_{\bowtie}(S)|+\frac{2}{9}t_1(S)+\frac{2}{3}t_2(S).
\end{equation}
Moreover, since $H$ is a connected graph, $E(H)\geq V(H)-1$, i.e.,
\begin{equation}\label{eq_4}
|4_{\bowtie}(S)|\geq |T(S)|-1.
\end{equation} 

If $H$ is not a tree, then Formula \ref{eq_4} can be strengthened as $|4_{\bowtie}(S)|\geq |T(S)|$, which together with
Formula \ref{eq_3} gives $|4_{\bowtie}(S)|\geq |3_{\Delta}(S)|+t_1(S)+t_2(S)$.
Hence, we can deduce from Formula \ref{eq_3-3} that
$ch^*(S)\geq \frac{8}{27}|3_{\Delta^+}(S)|+\frac{4}{27}|3_{\Delta^-}(S)|+\frac{4}{27}|3_{\Delta^{\star}}(S)|+\frac{2}{3}t_1(S)+\frac{10}{9}t_2(S)\geq 0$. Therefore, we may next assume that $H$ is a tree. 

We will first show that the following inequality holds: 
\begin{equation}\label{eq_8}
|3_{\Delta^+}(S)|+|3_{\Delta^-}(S)|+2t_1(S)+2t_2(S)\geq 4.
\end{equation}
Note that each leaf of $H$ corresponds to a 3-face of $S$ which contains a $C$-vertex or at least two $3_{\Delta^+}$- or $3_{\Delta^-}$-vertices. Formula \ref{eq_8} follows when $H$ is not an isolated vertex.
Next, let $H$ be an isolated vertex, i.e., $S$ is a 3-face.
By Lemma \ref{lem_type-a} and the assumption that $S$ is not a positive $(3,3,3)$-face,  $S$ contains a $C$-vertex.
It follows that $S$ contains either another $C$-vertex or two $3_{\Delta^+}$- or $3_{\Delta^-}$-vertices, yielding Formula \ref{eq_8} as well.

Combining Formulas \ref{eq_3} and \ref{eq_4} gives $|4_{\bowtie}(S)|\geq |3_{\Delta}(S)|+t_1(S)+t_2(S)-3.$
Hence, we can deduce from Formula \ref{eq_3-3} that
\begin{equation}\label{eq_11}
ch^*(S)\geq \frac{8}{27}|3_{\Delta^+}(S)|+\frac{4}{27}|3_{\Delta^-}(S)|+\frac{4}{27}|3_{\Delta^{\star}}(S)|+\frac{2}{3}t_1(S)+\frac{10}{9}t_2(S)-\frac{4}{3}.
\end{equation}

Clearly, both $3_{\Delta^+}(S)$ and $3_{\Delta^-}(S)$ are even integers, and if $t_1(S)=1$ then $t_2(S)\geq 1$. 
 So by applying Formula \ref{eq_8}, we can deduce from Formula \ref{eq_11} that $ch^*(S)\geq 0$ except the following six cases:
\begin{enumerate}[(a)]
	  \setlength{\itemsep}{0pt}
	\item $|3_{\Delta^+}(S)|=0$, $|3_{\Delta^-}(S)|=4$, $|3_{\Delta^{\star}}(S)|\leq 4$, $t_1(S)=t_2(S)=0$;	
	\item $|3_{\Delta^+}(S)|=2$, $|3_{\Delta^-}(S)|=2$, $|3_{\Delta^{\star}}(S)|\leq 2$, $t_1(S)=t_2(S)=0$;	
	\item $|3_{\Delta^+}(S)|=4$, $|3_{\Delta^-}(S)|=0$, $|3_{\Delta^{\star}}(S)|=0$, $t_1(S)=t_2(S)=0$;	
	\item $|3_{\Delta^+}(S)|=0$, $|3_{\Delta^-}(S)|=6$, $|3_{\Delta^{\star}}(S)|\leq 2$, $t_1(S)=t_2(S)=0$;
	\item $|3_{\Delta^+}(S)|=0$, $|3_{\Delta^-}(S)|=8$, $|3_{\Delta^{\star}}(S)|=0$, $t_1(S)=t_2(S)=0$;	
	\item $|3_{\Delta^+}(S)|=2$, $|3_{\Delta^-}(S)|=4$, $|3_{\Delta^{\star}}(S)|=0$, $t_1(S)=t_2(S)=0$.	
\end{enumerate}

Next we consider these exceptional cases, for which the calculation of $ch^*(S)$ will always be based on Formula \ref{eq_11}.
Since $t_1(S)=t_2(S)=0$, $S$ has no $C$-vertices. 
Recall that $H$ is a subcubic planar tree. So, $H$ contains two less vertices of degree 3 (in $H$) than leaves. Correspondingly, $S$ contains two less $(4,4,4)$-faces than $(3,3,4)$-faces. Moreover, the remaining 3-faces of $S$ must be $(3,4,4)$-faces.

For Cases a, b, and c: Since $|3_{\Delta^+}(S)|+|3_{\Delta^-}(S)|=4$, $H$ is a path. Denote by $k$ the length of $H$.
If $k=1$ (i.e., $H$ is an edge), then $S$ is Configuration b-1, which however is reducible by Lemma \ref{lem_type-b}. So, $k\geq 2$.
The snowflake $S$ can be labelled as follows: let $T_0=[uvw_1]$ and $T_k=[w_kxy]$ be two $(3,3,4)$-faces, and $T_i=[w_iz_iw_{i+1}]$ be a $(3,4,4)$-face for $1\leq i\leq k-1$ so that $T_j$ intersects with $T_{j+1}$ at the 4-vertex $w_{j+1}$ for $0\leq j\leq k-1$. Let $t_j$ be the outer neighbor of $z_j$.   

{\bf Case c:}
Since $G$ has no Configuration c-1 by Lemma \ref{lem_type-c}, $k\geq 7$.
By Lemma \ref{lem_J-segment}, the snowflake $S$ is related to at least four nice 9-faces, whose nice vertices send to $S$ a total charge at least $\frac{2}{27}\times 4$ by $R\ref{rule_nice face}$.
Therefore, it follows from Formula \ref{eq_11} that $ch^*(S)\geq \frac{8}{27}\times 4-\frac{4}{3}+\frac{2}{27}\times 4> 0.$

{\bf Case b:} W.l.o.g., let $T_0$ be negative.
Since $G$ has no Configuration d-1 by Lemma \ref{lem_type-d}, $T_1$ is negative.
Since $G$ has no Configuration b-2 by Lemma \ref{lem_type-b}, $k\notin \{2,3\}$.
Hence, $k\geq 4$.
Since $G$ has no Configuration e-3 by Lemma \ref{lem_type-e}, neither $t_1$ nor $t_2$ is a bad vertex. So, $|3_{\Delta^{\star}}(S)|=2$, i.e., all of $t_3,t_4,\ldots,t_{k-1}$ are bad vertices. 
Since $G$ has no Configuration c-2 by Lemma \ref{lem_type-c}, $k\geq 9$.
By Lemma \ref{lem_J-segment}, $S$ is related to at least four nice 9-faces, strengthening Formula \ref{eq_11} as $ch^*(S)\geq \frac{8}{27}\times 2+\frac{4}{27}\times 2+\frac{4}{27}\times 2-\frac{4}{3}+\frac{2}{27}\times 4> 0$ by $R\ref{rule_nice face}$.

{\bf Case a:} Similar to Case b, we can show that $k\geq 4$, both $T_{1}$ and $T_{k-1}$ are negative, and none of $t_1,t_2,t_{k-2},t_{k-1}$ are bad vertices.
Let $P=w_1w_2\ldots w_k$.

Case a.1: assume that $T_1, T_2, \ldots, T_{k-1}$ locate on the same side of $P$.
So, $P$ belongs to the boundary of some face, say $f=[uw_1w_2\ldots w_kys_1s_2\ldots s_n]$. 

Case a.1.1: let $k\geq 5$.
In this case, $|3_{\Delta^{\star}}(S)|=4$, i.e., $t_3,t_4,\ldots,t_{k-3}$ (if exist) are all bad vertices.
If $d(f)\geq 10$, then $f$ sends to each vertex of $uw_1w_2\ldots w_ky$ charge at least $\frac{6}{10}$ by $R\ref{rule-face}$, strengthening Formula \ref{eq_11} as $ch^*(S)\geq \frac{4}{27}\times 4+\frac{4}{27}\times 4-\frac{4}{3}+(\frac{6}{10}-\frac{5}{9})\times (2+k)>0.$ Next, let $d(f)=9$.
Then $k\in\{5,6,7\}$.
For $k=5$, since $G$ has no Configuration f-3 by Lemma \ref{lem_type-f}, at least one of $s_1$ and $s_2$ (w.l.o.g., say $s_1$) is not a $3_{\Delta}$-vertex. Then $y$ sends no charge to $s_1$ and instead, it receives charge $\frac{2}{27}$ from $s_1$, giving $ch^*(S)\geq \frac{4}{27}\times 4+\frac{4}{27}\times 4-\frac{4}{3}+(\frac{2}{27}+\frac{2}{27})=0.$
For $k\in\{6,7\}$, it is obvious that neither the outer neighbor of $u$ nor that of $y$ is a $3_{\Delta^+}$-vertex. So, both $u$ and $y$ send no charge to their outer neighbor, yielding $ch^*(S)\geq \frac{4}{27}\times 4+\frac{4}{27}\times 4-\frac{4}{3}+\frac{2}{27}\times 2=0$.

Case a.1.2: let $k=4$.
In this case, $t_2$ coincides with $t_{k-2}$, and thus $|3_{\Delta^{\star}}(S)|=3$.

If $d(f)\geq 11$, then $f$ sends to each vertex of $uw_1w_2w_3w_4y$ charge at least $\frac{7}{11}$ by $R\ref{rule-face}$, strengthening Formula \ref{eq_11} as $ch^*(S)\geq \frac{4}{27}\times 4+\frac{4}{27}\times 3-\frac{4}{3}+(\frac{7}{11}-\frac{5}{9})\times 6>0.$

Let $d(f)=10$. If both $s_1$ and $s_4$ are $3_{\Delta^+}$-vertices, say 3-faces $[s_1r_{12}s_2]$ and $[s_3r_{34}s_4]$, then at least one of $xys_1r_{12}$, $s_1s_2s_3s_4$, and $r_{34}s_4uv$ is Configuration e-2, contradicting Lemma \ref{lem_type-e}.
Hence, at least one of $s_1$ and $s_4$ is not a $3_{\Delta^+}$-vertex and correspondingly, at least one of $u$ and $y$ sends no charge to its outer neighbor. 
Moreover, $f$ sends to each vertex of $uw_1w_2w_3w_4y$ charge $\frac{6}{10}$ by $R\ref{rule-face}$.
Therefore, $ch^*(S)\geq \frac{4}{27}\times 4+\frac{4}{27}\times 3-\frac{4}{3}+\frac{2}{27}+(\frac{6}{10}-\frac{5}{9})\times 6>0.$

It remains to assume that $d(f)=9$.
If $s_1$ is a $4^+$-vertex or an external 3-vertex, i.e., $s_1$ is a 2-nice vertex of $f$, then $s_1$ sends charge $\frac{4}{27}$ to $S$ and $\frac{2}{27}$ to $y$, strengthening Formula \ref{eq_11} as  $ch^*(S)\geq \frac{4}{27}\times 4+\frac{4}{27}\times 3-\frac{4}{3}+\frac{4}{27}+(\frac{2}{27}+\frac{2}{27})=0.$ Hence, we may next assume that $s_1$ is an internal 3-vertex and similarly, so does $s_3$. Since $G$ has no Configuration f-1 by Lemma \ref{lem_type-f}, $d(s_2)\geq 4$. Then $s_1$ cannot be a $3_{\Delta^+}$-vertex, since otherwise $xys_1r_{12}$ is Configuration e-2. Similarly, neither does $s_3$. So, both $u$ and $y$ send no charge to their outer neighbors.
If $d(s_2)\neq 4$, i.e., $s_2$ is a 2-nice vertex of $f$, then $s_2$ sends charge $\frac{4}{27}$ to $S$, giving $ch^*(S)\geq \frac{4}{27}\times 4+\frac{4}{27}\times 3-\frac{4}{3}+\frac{2}{27}\times 2+\frac{4}{27}=0.$  
Next, let $d(s_2)=4$.
Since $G$ has no Configuration f-2 by Lemma \ref{lem_type-f}, neither $s_1$ nor $s_3$ is a $3_{\Delta^-}$-vertex. So, both $u$ and $y$ receive charge $\frac{2}{27}$ from their outer neighbors, giving $ch^*(S)\geq \frac{4}{27}\times 4+\frac{4}{27}\times 3-\frac{4}{3}+(\frac{2}{27}+\frac{2}{27})\times 2=0.$

Case a.2: assume that not all of $T_1, T_2, \ldots, T_{k-1}$ locate on the same side of $P$.
Recall that  both $T_{1}$ and $T_{k-1}$ are negative.
If $k=4$, then either $S$ is Configuration b-2 or $S$ contains Configuration d-2, contradicting Lemma \ref{lem_type-b} or \ref{lem_type-d}, respectively.
Hence, $k\geq 5$. 
It follows that $|3_{\Delta^{\star}}(S)|=4$, and thus $t_3,t_4,\ldots,t_{k-3}$ (if exist) are all bad vertices.

Case a.2.1: let $k=5$.
Since $G$ has none of Configurations b-2, b-3 and d-1, we can deduce that precisely one of $T_2$ and $T_3$ (w.l.o.g., say $T_3$) is positive, and $T_1$ locates on one side of $P$ while $T_2, T_3, T_4$ locate on the other side.
So, $t_1z_1w_2w_3w_4w_5$ belongs to the boundary of some face, say $f=[t_1z_1w_2w_3w_4w_5ys_1s_2\ldots]$. 
If $d(f)\geq 10$, then $f$ sends to each vertex of $z_1w_2w_3w_4w_5y$ charge at least $\frac{6}{10}$ by $R\ref{rule-face}$, strengthening Formula \ref{eq_11} as $ch^*(S)\geq \frac{4}{27}\times 4+\frac{4}{27}\times 4-\frac{4}{3}+(\frac{6}{10}-\frac{5}{9})\times 6>0.$
Next, let $d(f)=9$.
If $s_1$ is not a  $3_{\Delta^+}$- or $3_{\Delta^-}$-vertex, then $y$ receives charge $\frac{2}{27}$ from $s_1$ by $R\ref{rule-A1-vertex}$, strengthening Formula \ref{eq_11} as $ch^*(S)\geq \frac{4}{27}\times 4+\frac{4}{27}\times 4-\frac{4}{3}+(\frac{2}{27}+\frac{2}{27})=0.$ 
Next, let $s_1$ be a  $3_{\Delta^+}$- or $3_{\Delta^-}$-vertex.
Moreover, if at least one of $t_1$ and $s_2$ is a nice vertex of $f$, then it is a 2-nice vertex and sends charge $\frac{4}{27}$ to $S$, giving $ch^*(S)\geq \frac{4}{27}\times 4+\frac{4}{27}\times 4-\frac{4}{3}+\frac{4}{27}=0.$
Next, let $t_1$ be an internal 3-vertex and let $d(s_2)\leq 4$.
Since $G$ has no Configuration b-1, we can deduce that $t_1$ is not a  $3_{\Delta^+}$- or $3_{\Delta^-}$-vertex. 
Moreover, since $G$ has no Configuration e-2, we can deduce that $s_1$ cannot be a $3_{\Delta^+}$-vertex.
Now, both $y$ and $z_1$ send no charge to their outer neighbors, giving $ch^*(S)\geq \frac{4}{27}\times 4+\frac{4}{27}\times 4-\frac{4}{3}+\frac{2}{27}\times 2=0.$ 

Case a.2.2: let $k\geq 6$.
Since $G$ has no Configuration e-3 by Lemma \ref{lem_type-e}, both $T_2$ and $T_{k-2}$ are positive.
Since $G$ has neither Configuration d-2 nor Configuration c-3 by Lemma \ref{lem_type-d} or \ref{lem_type-c}, $T_1,T_2,T_{k-2},T_{k-1}$ locate on the same side of $P$.
Denote by $i$ the minimum index such that $T_i$ and $T_1$ locate on different sides of $P$, and $j$ the maximum one.
W.l.o.g., let $uw_1w_2\dots w_iz_it_i$ belong to the boundary of a face $f_1$, and $t_jz_jw_{j+1}w_{j+2}\ldots w_ky$ belong to the boundary of a face $f_2$.
If $d(f_1)\geq 10$, then $f_1$ sends to each vertex of $uw_1w_2\dots w_iz_i$ charge at least $\frac{6}{10}$ by $R\ref{rule-face}$, strengthening Formula \ref{eq_11} as $ch^*(S)\geq \frac{4}{27}\times 4+\frac{4}{27}\times 4-\frac{4}{3}+(\frac{6}{10}-\frac{5}{9})\times (2+i)>0$. We may next assume that $d(f_1)=9$ and similarly, $d(f_2)=9$. Notice that $3\leq i\leq j\leq k-3$. Since $G$ has no Configuration e-2, we can deduce that the outer neighbor of $u$ (resp., $y$) cannot be a $3_{\Delta^+}$-vertex and so it receives no charge from $u$ (resp., $y$), strengthening Formula \ref{eq_11} as $ch^*(S)\geq \frac{4}{27}\times 4+\frac{4}{27}\times 4-\frac{4}{3}+\frac{2}{27}\times 2=0$.

{\bf Case d}: Clearly, $H$ is a subdivision of a claw. Denote by $k_1,k_2,k_3$ ($1\leq k_1\leq k_2\leq k_3$) the length of paths between a leaf and the 3-vertex in $H$.
So, the snowflake $S$ can be labelled as follows:
let $T_0=[w_{11}w_{21}w_{31}]$ be a $(4,4,4)$-face, $T_{ik_i}=[w_{ik_i}x_iy_i]$ be a $(3,3,4)$-face, $T_{ij}=[w_{ij}z_{ij}w_{i,j+1}]$ be a $(3,4,4)$-face, $t_{ij}$ be the outer neighbor of $z_{ij}$ so that $T_{i,l-1}$ intersects with $T_{il}$ at the 4-vertex $w_{il}$ for $1\leq i\leq 3$, $1\leq j\leq k_i-1$, and $1\leq l\leq k_i$.
W.l.o.g., let $w_{11},w_{21},w_{31}$ locate in clockwise order around $T_0$, and so do $x_1,y_1,x_2,y_2,x_3,y_3$.

For $1\leq i\leq 3$, if $k_i\geq 2$, then the face $T_{i,k_i-1}$ is negative (since $G$ has no Configuration d-1 by Lemma \ref{lem_type-d}) and $t_{i,k_i-1}$ is not a bad vertex (since $G$ has no Configuration e-3 by Lemma \ref{lem_type-e}). Furthermore, if $k_i\geq 3$, then $t_{i,k_i-2}$ is not a bad vertex (again since $G$ has no Configuration e-3). 
Since $|3_{\Delta^{\star}}(S)|\leq 2$, it suffices to consider two cases: either $(k_1,k_2,k_3)=(1,1,k_3)$ or $(k_1,k_2,k_3)=(1,2,2)$.

Case d.1: assume that $(k_1,k_2,k_3)=(1,1,k_3)$. 
Since $G$ has neither Configuration g-1 nor Configuration g-2 by Lemma \ref{lem_type-g}, $k_3\geq 4$.
So, $t_{3,1},t_{3,2},\ldots,t_{3,k_3-3}$ are all bad vertices.
Since $G$ has no Configuration g-5 by Lemma \ref{lem_type-g}, $k_3\geq 8$.
By Lemma \ref{lem_J-segment}, the snowflake $S$ is related to at least three nice 9-faces, whose nice vertices send to $S$ a total charge at least $\frac{2}{27}\times 3$ by $R\ref{rule_nice face}$.
Therefore, $ch^*(S)\geq \frac{4}{27}\times 6+\frac{4}{27}\times 2-\frac{4}{3}+\frac{2}{27}\times 3>0.$

Case d.2: assume that $(k_1,k_2,k_3)=(1,2,2)$. 
Since $G$ has no Configuration g-3 by Lemma \ref{lem_type-g}, $T_{21}$ and $T_{31}$ must locate on the same side of $w_{22}w_{21}w_{31}w_{32}$. We distinguish two cases depending on which side.

Case d.2.1: let $y_2w_{22}w_{21}w_{31}w_{32}x_3$ belong to the boundary of some face $f_{23}$.
If $d(f_{23})\geq 10$, then $f_{23}$ sends to each vertex of $y_2w_{22}w_{21}w_{31}w_{32}x_3$ charge at least $\frac{6}{10}$ by $R\ref{rule-face}$, strengthening Formula \ref{eq_11} as $ch^*(S)\geq \frac{4}{27}\times 6+\frac{4}{27}\times 2-\frac{4}{3}+(\frac{6}{10}-\frac{5}{9})\times 6>0.$ We may next assume that $d(f_{23})=9$. Let $f_{23}=[y_2w_{22}w_{21}w_{31}w_{32}x_3s_1s_2s_3]$.
If neither $s_1$ nor $s_3$ is a $3_{\Delta^+}$-vertex, then both $x_3$ and $y_2$ send no charge to their outer neighbors, strengthening Formula \ref{eq_11} as $ch^*(S)\geq \frac{4}{27}\times 6+\frac{4}{27}\times 2-\frac{4}{3}+\frac{2}{27}\times 2=0.$ 
W.l.o.g., next let $s_1$ be a $3_{\Delta^+}$-vertex. Since $G$ has no Configuration e-2, $s_2$ must be a $3_{\Delta^+}$-vertex and hence, $s_3$ cannot be an internal 3-vertex. So, $y_2$ receives charge  $\frac{2}{27}$ from $s_3$, giving $ch^*(S)\geq \frac{4}{27}\times 6+\frac{4}{27}\times 2-\frac{4}{3}+(\frac{2}{27}+\frac{2}{27})=0.$

Case d.2.2: let $y_1w_{11}w_{21}w_{22}x_2$ (resp., $y_3w_{32}w_{31}w_{11}x_1$) belong to the boundary of some face $f_{12}$ (resp., $f_{13}$).
If $d(f_{12})\geq 10$, then $f_{12}$ sends to each vertex of $y_1w_{11}w_{21}w_{22}x_2$ charge at least $\frac{6}{10}$ by $R\ref{rule-face}$, strengthening Formula \ref{eq_11} as $ch^*(S)\geq \frac{4}{27}\times 6+\frac{4}{27}\times 2-\frac{4}{3}+(\frac{6}{10}-\frac{5}{9})\times 5>0.$ We may next assume that $d(f_{12})=9$ and similarly, $d(f_{13})=9$. Let $f_{12}=[y_1w_{11}w_{21}w_{22}x_2s_1s_2s_3s_4]$.
If both $s_1$ and $s_4$ are $3_{\Delta^+}$-vertices, say 3-faces $[s_1r_{12}s_2]$ and $[s_3r_{34}s_4]$.
It is easy to see that at least one of $y_2x_2s_1r_{12}$, $s_1s_2s_3s_4$, and $r_{34}s_4y_1x_1$ is Configuration e-2, contradicting Lemma \ref{lem_type-e}.
So, at least one of $s_1$ and $s_4$ is not a $3_{\Delta^+}$-vertex and consequently, at least one of $y_1$ and $x_2$ sends no charge to its outer neighbor.
Similarly, at least one of $x_1$ and $y_3$ sends no charge to its outer neighbor.
Therefore, Formula \ref{eq_11} can be strengthened as $ch^*(S)\geq \frac{4}{27}\times 6+\frac{4}{27}\times 2-\frac{4}{3}+\frac{2}{27}\times 2=0.$

{\bf Case e}: Clearly, $S$ consists of two $(4,4,4)$-faces, four negative $(3,3,4)$-faces, and  $(3,4,4)$-faces. Recall that $|3_{\Delta^{\star}}(S)|=0$.
Since $G$ has no Configuration e-3 by Lemma \ref{lem_type-e}, each $(4,4,4)$-face of $S$ must intersect with two negative $(3,3,4)$-faces.
Since $G$ has neither Configuration h-1 nor Configuration h-2 by Lemma \ref{lem_type-h}, $S$ must be a $J$-extension of Configuration h-1 for some $k\geq 4$.
By Lemma \ref{lem_J-segment}, $S$ is related to at least two nice 9-faces, whose nice vertices send to $S$ a total charge at least $\frac{2}{27}\times 2$ by $R\ref{rule_nice face}$.
Therefore, $ch^*(S)\geq \frac{4}{27}\times 8-\frac{4}{3}+\frac{2}{27}\times 2= 0.$

{\bf Case f}: We apply a similar argument as Case e. Clearly,
$S$ consists of one $(4,4,4)$-face, one positive $(3,3,4)$-face, two negative $(3,3,4)$-faces, and $(3,4,4)$-faces.
Since $G$ has no Configuration e-3 by Lemma \ref{lem_type-e}, the $(4,4,4)$-face of $S$ intersects with two negative $(3,3,4)$-faces.
Since $G$ has neither Configuration g-1 nor Configuration g-4 by Lemmas \ref{lem_type-g}, $S$ must be a $J_k$-extension of Configuration g-1 for some $k\geq 5$.
By Lemma \ref{lem_J-segment}, the snowflake $S$ is related to at least three nice 9-faces, whose nice vertices send to $S$ a total charge at least $\frac{2}{27}\times 3$ by $R\ref{rule_nice face}$.
Therefore, $ch^*(S)\geq \frac{8}{27}\times 2+\frac{4}{27}\times 4-\frac{4}{3}+\frac{2}{27}\times 3> 0.$
\end{proof}

\begin{claim}\label{claim_int_vertex}
	$ch^*(v)\geq 0$ for each internal $C$-vertex $v$ of $G$.
\end{claim}

\begin{proof}
Since $v$ is internal, $v$ is adjacent to no strings of $G$. So, $R\ref{rule-string}$ is not applicable to $v$. 
Denote by $n_3(v)$ and $n_9(v)$ the number of 3-faces and $9^+$-faces containing $v$, respectively. Let $n(v)$ be the number of $3_{\Delta}$-vertices whose outer neighbor is $v$.
Since $G\in \mathcal{G}$,
\begin{equation}\label{eq_row} 
\begin{split}
n_9(v)&\geq n(v)+n_3(v),  \\
d(v)&\geq n(v)+2n_3(v).
\end{split}
\end{equation}
The vertex $v$ sends charge 1 to each incident 3-face by $R$\ref{rule-C-vertex}, and charge at most $\frac{2}{9}$ to each $3_{\Delta}$-vertex whose outer neighbor is $v$ by $R$\ref{rule-A1-vertex}. Moreover, $v$ receives charge at least $\frac{5}{9}$ from each incident $9^+$-face by $R$\ref{rule-face}. Finally, recall that a nice 9-face is related to precisely one snowflake. If $v$ is a nice vertex of a nice 9-face $f$, then $v$ sends to the related snowflake of $f$ charge at most $\frac{4}{27}$ by $R\ref{rule_nice face}$. Therefore, we can conclude from above that
\begin{equation}\label{eq_ch_v}
\begin{split}
ch^*(v)&\geq d(v)-4-n_3(v)-\frac{2}{9}n(v)+(\frac{5}{9}-\frac{4}{27})n_9(v) \\
&=d(v)-4+\frac{2}{27}n(v)+\frac{11}{27}(n_9(v)-n_3(v))-\frac{8}{27}(n(v)+2n_3(v))\\ 
&\geq d(v)-4+\frac{2}{27}n(v)+\frac{11}{27}n(v)-\frac{8}{27}d(v)\\
&=\frac{19}{27}d(v)-4+\frac{13}{27}n(v), 
\end{split}
\end{equation}
where the second inequality uses Formula \ref{eq_row}.
By Lemma \ref{lem_min degree}, $d(v)\geq 3$.  We distinguish the following three cases.

{\bf Case 1}: assume that $d(v)\geq 5$.
 The last line of Formula \ref{eq_ch_v} gives $ch^*(v)\geq 0$ directly except when $d(v)=5$ and $n(v)=0$. 
 For this exceptional case, $n_3(v) \leq 2$, and $n_3(v)=2$ implies $n_9(v)=3$.
 So, the first line of Formula \ref{eq_ch_v} yields $ch^*(v)\geq 1-n_3(v)+\frac{11}{27}n_9(v)\geq 0$.

{\bf Case 2}: assume that $d(v)=4$. 
Since $n(v)\leq n_9(v)$ by Formula $\ref{eq_row}$, if $n_3(v)=0$, then the first line of Formula \ref{eq_ch_v} gives $ch^*(v)\geq -\frac{2}{9}n(v)+\frac{11}{27}n_9(v)\geq 0$.
Since $v$ is not a $4_{\bowtie}$-vertex, it remains to assume that $n_3(v)=1$.  
Denote by $f_1,f_2,f_3,f_4$ the faces containing $v$, and $w,x,y,z$ the neighbors of $v$, in the same cyclic order with $f_4=[vwx]$. Clearly, both $f_1$ and $f_3$ are $9^+$-faces. 
If $n(v)=0$, then $v$ cannot be a nice vertex by definition, strengthening the first line of Formula \ref{eq_ch_v} as $ch^*(v)\geq -1+\frac{5}{9}\times 2>0$.
Next, let $n(v)\in\{1,2\}$. Then $f_2$ is also a $9^+$-face.
Note that if $y$ (resp., $z$) is a bad vertex, then $v$ cannot be a 2-nice vertex of $f_1$ (resp., $f_3$). Also note that if both $y$ and $z$ are bad vertices, then $v$ cannot be a 2-nice vertex of $f_2$. Considering charges $v$ receives from incident $9^+$-faces and charges $v$ sends to $f_4$, to $y$ and $z$, and to snowflakes related to $f_1$, $f_2$ or $f_3$, we have $ch^*(v)\geq \frac{5}{9}\times 3-1-\max\{\frac{2}{9}\times 2+\frac{2}{27}\times 3, (\frac{2}{9}+\frac{2}{27})+(\frac{4}{27}\times 2+\frac{2}{27}), \frac{2}{27}\times 2+\frac{4}{27}\times 3\}=0$.

{\bf Case 3}: assume that $d(v)=3$. 
In this case, $n_3(v)=0$ and $v$ cannot be a nice vertex. Clearly, the faces incident with $v$ are one $5^+$-face and two $7^+$-faces if $n(v)=0$,
one $5^+$-face and two $9^+$-faces if $n(v)=1$, and three $9^+$-faces if $n(v)\geq 2$.
Therefore,
$ch^*(v)\geq -1+\min\{\frac{1}{5}+\frac{3}{7}\times 2, -\frac{2}{9}+\frac{1}{5}+\frac{5}{9}\times 2, -\frac{2}{9}\times 3+\frac{5}{9}\times 3\}=0$ by $R\ref{rule-C-vertex}$ and $R\ref{rule-face}$.
\end{proof}

\begin{claim}
	$ch^*(v)>0$ for each external $C$-vertex $v$ of $G$.
\end{claim}

\begin{proof}
By Lemma \ref{lem_string} and the rule $R$\ref{rule-string}, $v$ sends to the vertices of each adjacent string a total charge at most 
\begin{equation}\label{eq_charge2s}
\begin{split}
\frac{12-d(f)}{6d(f)}\times (\lfloor \frac{d(f)-1}{2} \rfloor-1)&\leq \frac{(12-d(f))(d(f)-3)}{12d(f)}=\frac{5}{4}-(\frac{d(f)}{12}+\frac{3}{d(f)})\\
&\leq \frac{5}{4}-2\sqrt{\frac{d(f)}{12}\times \frac{3}{d(f)}}= \frac{1}{4}, 
\end{split}
\end{equation}  
where $f$ is the face (other than $f_0$) containing this string.

Take the notation $n_3(v)$, $n_9(v)$, and $n(v)$ for the same meaning as in the proof of Claim \ref{claim_int_vertex}. So, Formula \ref{eq_row} still holds.
However, Formula \ref{eq_ch_v} should be changed to 
\begin{equation}\label{eq_ch_v_ext}
	\begin{split}
		ch^*(v)&\geq d(v)-4-n_3(v)-\frac{2}{9}n(v)+(\frac{5}{9}-\frac{4}{27})(n_9(v)-1)+\frac{4}{3}-\frac{1}{4}\times 2\\
            &\geq d(v)-4-\frac{16}{27}n_3(v)+\frac{5}{27}n(v)+\frac{23}{54} \quad \quad \text{(by using $n_9(v)\geq n(v)+n_3(v)$)}\\
		&\geq \frac{19}{27}d(v)-4+\frac{13}{27}n(v)+\frac{23}{54}. \quad \quad \text{(by using $n_3(v)\leq \frac{d(v)-n(v)}{2}$)} 
	\end{split}
\end{equation}
This is because $v$ receives from $f_0$ charge $\frac{4}{3}$ and additionally, $v$ sends a total charge at most $\frac{1}{4}\times 2$ to the vertices of adjacent strings by Formula \ref{eq_charge2s}.

Therefore, the conclusion $\ch^*(v)\geq 0$ follows directly from the last line of Formula \ref{eq_ch_v_ext} when either $d(v)\geq 6$ or $d(v)=5$ and $n(v)\geq 1$, follows from the second line of Formula \ref{eq_ch_v_ext} when either $(d(v),n(v))=(5,0)$ (since $n_3(v)\leq 2$ in this case) or $(d(v),n_3(v))=(4,0)$, and follows from the first line of Formula \ref{eq_ch_v_ext} when $(d(v),n_3(v))=(4,1)$ (since $n_9(v)\geq \min\{2+n(v),3\}$ in this case). So, it remains to consider the following two cases:

Firstly, assume that $(d(v),n_3(v))=(4,2)$. In this case, $v$ cannot be a nice vertex by definition.
If $d(f_0)=3$, then $v$ receives charge $\frac{4}{3}$ from $f_0$ and charge $\frac{5}{9}$ from each incident $9^+$-face, and $v$ sends charge $1$ to the other incident 3-face and charge at most $\frac{1}{4}$ to each adjacent string, giving $ch^*(v)\geq d(v)-4+\frac{4}{3}+\frac{5}{9}\times 2-1-\frac{1}{4}\times 2>0$. We may next assume that $d(f_0)\neq 3$, i.e., $f_0$ is one of the $9^+$-faces containing $v$. So, $v$ receives charge $\frac{4}{3}$ from $f_0$ and charge $\frac{5}{9}$ from the other incident $9^+$-face, and $v$ sends to each incident 3-face charge $\frac{5}{9}$ by $R\ref{rule-C-vertex}$, giving $ch^*(v)\geq d(v)-4+\frac{4}{3}+\frac{5}{9}-\frac{5}{9}\times 2>0$.

Secondly, assume that $d(v)=3$. 
Denote by $f_1$ and $f_2$ the two faces (besides $f_0$) containing $v$ with $d(f_1)\leq d(f_2)$. 
First let $d(f_1)=3$. Then $d(f_2)\geq 9$ and $v$ is not a nice vertex.
Note that $v$ receives charge $\frac{4}{3}$ from $f_0$ by $R\ref{rule-ext-face}$ and charge at least $\frac{5}{9}$ from $f_2$ by $R\ref{rule-face}$, and $v$ sends charge $\frac{5}{9}$ to $f_1$ by $R\ref{rule-C-vertex}$ and charge at most $\frac{1}{4}$ to one adjacent string (if exists) by $R$\ref{rule-string}, giving
$ch^*(v)\geq d(v)-4+\frac{4}{3}+\frac{5}{9}-\frac{5}{9}-\frac{1}{4}>0$.
Next let $d(f_1)\geq 5$. Since $G\in \mathcal{G}$, $f_2$ is a $7^+$-face. So, $v$ receives a total charge at least $\frac{4}{3}+\frac{1}{5}+\frac{3}{7}$ from incident faces. Moreover, $v$ sends to adjacent strings (if exist) a total charge at most $\frac{1}{4}\times 2$.
If the internal neighbor of $v$ is not a $3_{\Delta}$-vertex, then $v$ sends no charge to this neighbor and $v$ is not a nice vertex of $f_1$ or $f_2$, yielding $ch^*(v)\geq d(v)-4+\frac{4}{3}+\frac{1}{5}+\frac{3}{7}-\frac{1}{4}\times 2>0$; 
otherwise, both $f_1$ and $f_2$ are $9^+$-faces, yielding $ch^*(v)\geq d(v)-4+\frac{4}{3}+\frac{5}{9}+\frac{5}{9}-\frac{1}{4}\times 2-\frac{2}{9}-\frac{4}{27}\times 2>0$.
\end{proof}

\begin{claim}
	$ch^*(v)\geq 0$ for each 2-vertex $v$ of $G$. 
\end{claim}
\begin{proof}
	Let $f$ be the face containing $v$ other than $f_0$.
	Clearly, $v$ receives charge $\frac{4}{3}$ from $f_0$ by $R$\ref{rule-ext-face} and charge $\frac{d(f)-4}{d(f)}$ from $f$ by $R$\ref{rule-face}, which gives $ch^*(v)\geq d(v)-4+\frac{4}{3}+\frac{d(f)-4}{d(f)}=\frac{d(f)-12}{3d(f)}\geq 0$, provided by $d(f)\geq 12$.
	Next, let $d(f)\leq 11$. Denote by $L$ the string containing $v$.
	Lemma \ref{lem_string} implies that the two vertices adjacent to $L$ do not coincide, and they send to $v$ a total charge $\frac{12-d(f)}{6d(f)}\times 2$ by $R$\ref{rule-string}. So, $ch^*(v)\geq \frac{d(f)-12}{3d(f)}+\frac{12-d(f)}{6d(f)}\times 2=0.$
\end{proof}

\begin{claim}\label{claim_f_0}
	$ch^*(f_0)\geq 0$.
\end{claim}
\begin{proof}
	Since $R$\ref{rule-ext-face} is the only rule making $f_0$ move charge out, we have $ch^*(f_0)=d(f_0)+4-\frac{4}{3}\times d(f_0)=4-\frac{d(f_0)}{3}\geq 0$, since $d(f_0)\leq 12$.
\end{proof}

\begin{claim}\label{claim_f}
	$ch^*(f)\geq 0$ for each {$5^+$-face} $f$ of $G$ with $f\neq f_0$.
\end{claim}  
\begin{proof}
	Since $R$\ref{rule-face} is the only rule making $f$ move charge out, we have $ch^*(f)=d(f)-4-\frac{d(f)-4}{d(f)}\times d(f)=0.$
\end{proof}

As a counterexample to Theorem \ref{thm_main_extension-signed}, $(G,\sigma)$ must contain an external $C$-vertex. So by Claims \ref{claim_snow}--\ref{claim_f}, we have $\sum_{x\in V\cup F}ch^*(x)>0$, completing the proof of Theorem \ref{thm_main_extension-signed}.

\section{Acknowledgement}
Ligang Jin is supported by NSFC 11801522, U20A2068.
Yingli Kang is supported by NSFC 11901258 and ZJNSF LY22A010016.
Xuding Zhu is supported by  {NSFC 12371359, U20A2068.}


\begin{thebibliography}{99}
	\addtolength{\itemsep}{-1.5ex} 
	\small
			\bibitem{AbbottZhou1991203} {H. L. Abbott and B. Zhou},
			On small faces in 4-critical graphs,
			Ars Combin. 32 (1991) 203-207.
			

                \bibitem{Borodin2013} {O. V. Borodin}, Colorings of plane graphs: a survey, Discrete Math. 313 (2013) 517–539.
						
			\bibitem{4567} {O. V. Borodin, A. N. Glebov, A. Raspaud, and M. R. Salavatipour},
			Planar graphs without cycles of length from 4 to 7 are 3-colorable,
			J. Combin. Theory Ser. B 93 (2005) 303-311.	

   
												
			\bibitem{CA-disprove} {V. Cohen-Addad, M. Hebdige, D. Kr\'{a}l, Z. Li, and E. Salgado}, Steinberg's Conjecture is false, J. Combin. Theory Ser. B 122 (2017) 452-456.	
	
\bibitem{Dvorak-Postle-2018} Z. Dvo\v{r}\'{a}k and L. Postle, Correspondence coloring and its application to list-coloring planar graphs without cycles of lengths 4 to 8, J. Combin. Theory Ser. B 129 (2018) 38-54.	
			
	
	
			

\bibitem{Liu-loeb-Yin-Yu} R. Liu, S. Loeb, Y. Lin, and G. Yu, DP-3-coloring of some planar graphs, Discrete Math. 342 (1) (2019) 178-189.			


		\bibitem{Steinberg1993211} R. Steinberg,
		The state of the three color problem,
		in: J. Gimbel, J. W. Kennedy \& L. V. Quintas (eds.), Quo Vadis, Graph Theory? Ann Discrete Math 55 (1993) 211-248.
			
\bibitem{Voigt2007} M. Voigt, A non-3-choosable planar graph without cycles of length 4 and 5, Discrete Math. 307 (2007) 1013-1015.


			
\end{thebibliography}
\end{document}